\documentclass[reqno]{amsart}
\usepackage[left=2cm,right=2cm,top=2cm,bottom=1.5cm]{geometry}
\usepackage{longtable} 
\usepackage{amsmath,amssymb,amsthm}
\usepackage{enumerate} 
\usepackage{tikz}
\usepackage{array}
\newcolumntype{C}[1]{>{\centering\arraybackslash}m{#1}}
\newcolumntype{L}[1]{>{\raggedright\arraybackslash}m{#1}}

\newtheorem{theorem}{Theorem}[section]
\newtheorem{lemma}[theorem]{Lemma}
\newtheorem{proposition}[theorem]{Proposition}
\newtheorem{corollary}[theorem]{Corollary}
\theoremstyle{definition}
\newtheorem{definition}[theorem]{Definition}

\theoremstyle{remark}

\newtheorem{remark}[theorem]{Remark}

\numberwithin{equation}{section}

\newcommand{\Ann}[1]{{#1}^0}

\newcommand{\hook}{\lrcorner \,}
\newcommand{\id}{\mathrm{id}}
\newcommand{\ad}{\mathrm{ad}}
\newcommand{\vol}{\mathrm{vol}}

\newcommand{\bN}{\mathbb{N}} 

\newcommand{\bO}{\mathbb{O}} 
\newcommand{\bC}{\mathbb{C}}
\newcommand{\bR}{\mathbb{R}} 
\newcommand{\bH}{\mathbb{H}} 
\newcommand{\bF}{\mathbb{F}}

\renewcommand{\Re}{\mathrm{Re}} 
\renewcommand{\Im}{\mathrm{Im}}
 
\renewcommand{\dim}[1]{\mathrm{dim}(#1)}
\newcommand{\spa}[1]{\mathrm{span}(#1)}

\newcommand{\g}{\mathfrak{g}} 
\newcommand{\h}{\mathfrak{h}}
\newcommand{\uf}{\mathfrak{u}} 

\newcommand{\GL}{\mathrm{GL}} 
\newcommand{\SL}{\mathrm{SL}}
\newcommand{\Sp}{\mathrm{Sp}}
\newcommand{\sgn}{\mathrm{sgn}}
\newcommand{\Hom}{\mathrm{Hom}}
\newcommand{\JB}{\mathrm{JB}}

\newcommand{\G}{\mathrm{G}}
\newcommand{\Spin}{\mathrm{Spin}}
\newcommand{\End}{\mathrm{End}}
\newcommand{\PGL}{\mathrm{PGL}}

\newcommand{\op}{\oplus} 

\newcommand{\tr}{\mathrm{tr}}
 
\renewcommand{\a}{\alpha}

\newcommand{\vp}{\varphi} 
\renewcommand{\L}{\Lambda}

\begin{document}
\title{Cocalibrated structures on Lie algebras with a codimension one Abelian ideal}
\author{Marco Freibert}
\address{Marco Freibert, Fachbereich Mathematik, Universit\"at Hamburg, Bundesstra{\ss}e 55, 
D-20146 Hamburg, Germany}
\email{freibert@math.uni-hamburg.de}



\begin{abstract}
  Cocalibrated $\G_2$-structures and cocalibrated $\G_2^*$-structures are the natural
  initial values for Hitchin's evolution equations whose solutions
  define (pseudo)-Riemannian manifolds with holonomy group contained in $\Spin(7)$ or $\Spin_0(3,4)$, respectively.
  In this article, we classify which seven-dimensional real Lie algebras with a
  codimension one Abelian ideal admit such structures. Moreover, we classify the seven-dimensional
  complex Lie algebras with a codimension one Abelian ideal which admit cocalibrated $(\G_2)_{\bC}$-structures.
\end{abstract}

\maketitle

\section{Introduction}
In Berger's list \cite{Be} of possible holonomy groups of irreducible non-symmetric simply connected Riemannian manifolds the
case of holonomy equal to $\Spin(7)$ had been unsettled for over $30$ years. Finally, Bryant \cite{Br} in 1987 was the first who was able to construct examples of Riemannian manifolds with holonomy group equal to $\Spin(7)$. Moreover, he also constructed in the same article examples for the previously unsettled case of holonomy group equal to $\Spin_0(3,4)$, corresponding to irreducible non-symmetric simply connected pseudo-Riemannian metrics with split signature $(4,4)$. 

In 2000, Hitchin \cite{Hi} came up with a nice machinery to construct Riemannian manifolds with holonomy groups contained in $\Spin(7)$. Therefore he introduced a certain partial differential equation for three-forms on a compact seven-dimensional manifold $M$. He showed that solutions of this equation may be used to define a Riemannian metric with holonomy contained in $\Spin(7)$ on $M\times I$ for a sufficiently small interval $I\subseteq \bR$ if one uses as initial value what is called a \emph{cocalibrated $\G_2$-structure} $\varphi\in \Omega^3 M$ on $M$. A cocalibrated $\G_2$-structure $\varphi\in \Omega^3 M$ is a three-form on $M$ which pointwise looks like a certain standard form and whose Hodge dual $\star_{\varphi}\varphi\in \Omega^4 M$ with respect to the induced Riemannian metric is closed. In \cite{CLSS}, it has been proved that compactness is not necessary to get, using the machinery of Hitchin, a metric with holonomy contained in $\Spin(7)$ on an open neighborhood of $M\times \{0\}$ in $M\times \bR$. Moreover, Hitchin's flow equation has been generalized in \cite{CLSS} to the pseudo-Riemannian case, yielding a metric with holonomy contained in $\Spin_0(3,4)$ if the initial value is a cocalibrated $\G_2^*$-structure $\varphi\in \Omega^3 M$. Finally, it has been showed that if all data are real-analytic, a unique local solution exists in both cases. Similar to cocalibrated $\G_2$-structures, cocalibrated $\G_2^*$-structures $\varphi\in \Omega^3 M$ are three-forms which pointwise look like a certain other standard form and whose Hodge dual $\star_{\varphi} \varphi\in \Omega^4 M$ with respect to the induced pseudo-Riemannian metric of signature $(3,4)$ is closed.

This motivates to classify first the seven-dimensional real-analytic manifolds that admit real-analytic cocalibrated $\G_2$- or $\G_2^*$-structures, at least for certain simple subclasses. In a next step, it is then of interest to solve concretely Hitchin's flow equations on these manifolds and compute the holonomy group or give more abstract arguments that the holonomy group is further reduced to a subgroup of $\Spin(7)$ or $\Spin_0(3,4)$, respectively.

A classification of cocalibrated $\G_2$-structures has been
carried out for seven-dimensional compact homogeneous spaces with homogeneous cocalibrated
$\G_2$-structures \cite{R}. We focus on a similar case, namely cocalibrated left-invariant $\G_2$- or $\G_2^*$-structures
on Lie groups $\G$. In this case, $d\star_{\varphi} \varphi =0$ reduces to an algebraic equation
depending only on the structure coefficients of the Lie algebra $\g$ associated to $\G$. Such a structure is then
also called a cocalibrated $\G_2$- or $\G_2^*$-structure on a seven-dimensional Lie algebra.

Due to the vast amount of seven-dimensional real Lie algebras we restrict ourselves further and consider
cocalibrated $\G_2$- and $\G_2^*$-structures on seven dimensional real Lie algebras $\g$ which admit a
six-dimensional Abelian ideal $\uf$. For these Lie algebras, the entire Lie bracket is encoded
in the action of an element $e_7\in \g\backslash \uf$ on $\uf$ and the Lie algebras can be classified
by the complex Jordan normal form of $\ad(e_7)|_{\uf}$. This simple structure allows us to give a full classification of the Lie algebras
in question admitting such structures only in properties of the complex Jordan normal form of $\ad(e_7)|_{\uf}$. We also
classify which seven-dimensional complex Lie algebras $\g$ with complex codimension one Abelian ideal $\uf$ admit cocalibrated $(\G_2)_{\bC}$-structures. 
In a follow-up paper we will solve Hitchin's flow equations concretely for some of those seven-dimensional real Lie algebras, compute the holonomy groups in these cases and prove some general theorems on the possible holonomy groups.

Coming to our main results, let $\g$ be a real or complex seven-dimensional Lie algebra with codimension one Abelian ideal and number consecutively the diagonal elements of a fixed complex Jordan normal form by $\lambda_1,\ldots,\lambda_6$ and the Jordan blocks of the complex Jordan normal form by $1,\ldots, m$, both from the upper left to the lower right. Moreover, denote by $\JB(i)$ for all $i=1,\ldots, 6$, the number of the Jordan block in which the corresponding generalized eigenvector lies. Then our four main theorems can be stated as follows:
\begin{theorem}\label{Th1real}
Let $\g$ be a seven-dimensional real Lie algebra which admits a codimension
one Abelian ideal $\uf$. Let $e_7\in \g\backslash \uf$ and $\omega\in \L^2 \uf^*$ be a non-degenerated two-form
on $\uf$. Then the following are equivalent:
\begin{enumerate}
\item[(a)]
$\g$ admits a cocalibrated $\G_2$-structure.
\item[(b)]
$\g$ admits a cocalibrated $\G_2^*$-structure such that the subspace $\uf$ is non-degenerated with respect to the induced pseudo-Euclidean metric on $\g$.
\item[(c)]
$\ad(e_7)|_{\uf}\in \mathfrak{gl}(\uf)$ is similar under $\GL(\uf)$ to an element in $\mathfrak{sp}(\uf,\omega)$.
\item[(d)]
The complex Jordan normal form of $\ad(e_7)|_{\uf}$ has the property that for all $m\in \mathbb{N}$ and all $\lambda\neq 0$ the number of Jordan blocks of size $m$ with $\lambda$ on the diagonal is the same as the number of Jordan blocks of size $m$ with $-\lambda$ on the diagonal and the number of Jordan blocks of size $2m-1$ with $0$ on the diagonal is even.
\end{enumerate}
\end{theorem}
\begin{theorem}\label{Th1complex}
Let $\g$ be a seven-dimensional complex Lie algebra which admits a codimension
one Abelian ideal $\uf$. Let $e_7\in \g\backslash \uf$ and $\omega\in \L^2 \uf^*$ be a non-degenerated two-form
on $\uf$. Then the following are equivalent:
\begin{enumerate}
\item[(a)]
$\g$ admits a cocalibrated $(\G_2)_{\bC}$-structure such that the subspace $\uf$ is non-degenerated with respect to the induced non-degenerated complex symmetric bilinear form on $\g$.
\item[(b)]
$\ad(e_7)|_{\uf}\in \mathfrak{gl}(\uf)$ is similar under $\GL(\uf)$ to an element in $\mathfrak{sp}(\uf,\omega)$.
\item[(c)]
The complex Jordan normal form of $\ad(e_7)|_{\uf}$ has the property that for all $m\in \mathbb{N}$ and all $\lambda\neq 0$ the number of Jordan blocks of size $m$ with $\lambda$ on the diagonal is the same as the number of Jordan blocks of size $m$ with $-\lambda$ on the diagonal and the number of Jordan blocks of size $2m-1$ with $0$ on the diagonal is even.
\end{enumerate}
\end{theorem}
\begin{theorem}\label{Th2real}
Let $\g$ be a seven-dimensional real Lie algebra which admits a codimension
one Abelian ideal. Let $e_7\in \g\backslash \uf$, $V_2$ be a two-dimensional subspace of $\uf$, $V_4$ be a four-dimensional subspace of $\uf$ such that $\uf=V_2\oplus V_4$ and $\omega\in \L^2 V_4^*$ be a non-degenerated two-form on $V_4$. Then the following are equivalent:
\begin{enumerate}
\item[(a)]
$\g$ admits a cocalibrated $\G_2^*$-structure. 
\item[(b)]
$\ad(e_7)|_{\uf}\in \mathfrak{gl}(\uf)$ is similar under $\GL(\uf)$ to an element in
\begin{equation*}
\{f\in \mathfrak{gl}(\uf)| f|_{V_2}=f_2+h,\, f_2\in \mathfrak{gl}(V_2),h\in \Hom(V_2,V_4),\,\, f|_{V_4}=-\frac{\tr(f_2)}{2}\id_{V_4}+f_4,\,f_4\in \mathfrak{sp} (V_4,\omega_2)\}.
\end{equation*}
\item[(c)]
There exists a partition of $\{1,\ldots,6\}$ into three subsets $I_1$, $I_2$, $I_3$, each of cardinality two, such that the following is true:
\begin{enumerate}
\item[(i)]
$\sum_{i\in I_1} \lambda_i=\sum_{i\in I_2} \lambda_i=-\sum_{i\in I_3} \lambda_i$.
\item[(ii)] 
If there are $i_1\in  I_1$, $i_2\in I_2$ such that $\JB(i_1)=\JB(i_2)$ then $\lambda_{i_1}=\lambda_{i_2}=-\frac{\sum_{i\in I_3} \lambda_i}{2}$ or $\JB(j_1)=\JB(j_2)$ for the uniquely determined $j_k\in I_k$ such that $\{i_k,j_k\}=I_k$, $k=1,2$.
\item[(iii)]
If there exists $i_2\in I_2$ such that $\JB(j)=\JB(i_2)$ for all $j\in I_1$ or if there exists $i_1\in I_1$ such that $\JB(j)=\JB(i_1)$ for all $j\in I_2$, then $\lambda_j=-\frac{\sum_{i\in I_3} \lambda_i}{2}$ for all $j\in I_1\cup I_2$ and $\JB(j)=\JB(k)$ for all $j,k\in I_1\cup I_2$.
\end{enumerate}
\end{enumerate}
\end{theorem}
\begin{theorem}\label{Th2complex}
Let $\g$ be a seven-dimensional complex Lie algebra which admits a codimension
one Abelian ideal. Let $e_7\in \g\backslash \uf$, $V_2$ be a two-dimensional subspace of $\uf$, $V_4$ be a four-dimensional subspace of $\uf$ such that $\uf=V_2\oplus V_4$ and $\omega\in \L^2 V_4^*$ be a non-degenerated two-form on $V_4$. Then the following are equivalent:
\begin{enumerate}
\item[(a)]
$\g$ admits a cocalibrated $(\G_2)_{\bC}$-structure. 
\item[(b)]
$\ad(e_7)|_{\uf}\in \mathfrak{gl}(\uf)$ is similar under $\GL(\uf)$ to an element in
\begin{equation*}
\{f\in \mathfrak{gl}(\uf)| f|_{V_2}=f_2+h,\, f_2\in \mathfrak{gl}(V_2),h\in \Hom(V_2,V_4),\,\, f|_{V_4}=-\frac{\tr(f_2)}{2}\id_{V_4}+f_4,\,f_4\in \mathfrak{sp} (V_4,\omega_2)\}.
\end{equation*}
\item[(c)]
There exists a partition of $\{1,\ldots,6\}$ into three subsets $I_1$, $I_2$, $I_3$, each of cardinality two, such that the following is true:
\begin{enumerate}
\item[(i)]
$\sum_{i\in I_1} \lambda_i=\sum_{i\in I_2} \lambda_i=-\sum_{i\in I_3} \lambda_i$.
\item[(ii)] 
If there are $i_1\in  I_1$, $i_2\in I_2$ such that $\JB(i_1)=\JB(i_2)$ then $\lambda_{i_1}=\lambda_{i_2}=-\frac{\sum_{i\in I_3} \lambda_i}{2}$ or $\JB(j_1)=\JB(j_2)$ for the uniquely determined $j_k\in I_k$ such that $\{i_k,j_k\}=I_k$, $k=1,2$.
\item[(iii)]
If there exists $i_2\in I_2$ such that $\JB(j)=\JB(i_2)$ for all $j\in I_1$ or if there exists $i_1\in I_1$ such that $\JB(j)=\JB(i_1)$ for all $j\in I_2$, then $\lambda_j=-\frac{\sum_{i\in I_3} \lambda_i}{2}$ for all $j\in I_1\cup I_2$ and $\JB(j)=\JB(k)$ for all $j,k\in I_1\cup I_2$.
\end{enumerate}
\end{enumerate}
\end{theorem}
We like to point out two interesting consequences of the Theorems \ref{Th1real} - \ref{Th2complex}:

Firstly, Theorem \ref{Th1real} states that if a real seven-dimensional Lie algebra with codimension one Abelian ideal admits a cocalibrated $\G_2$-structure, then it also admits a cocalibrated $\G_2^*$-structure. It would be of interest to know if this is a general feature of seven-dimensional Lie algebras. We do not think so but are not able to give a concrete counterexample.

Secondly, the conditions for the existence of a cocalibrated $\G_2^*$-structure on a real seven-dimensional Lie algebra with codimension one Abelian ideal are completely analogously to the ones for a $(\G_2)_{\bC}$-structure on a complex seven-dimensional Lie algebra with codimension one Abelian ideal. In particular, a real seven-dimensional Lie algebra with codimension one Abelian ideal admits a cocalibrated $\G_2^*$-structure if and only if its complexification admits a cocalibrated $(\G_2)_{\bC}$-structure. Again it would be of interest to know if this is a general feature.
 
The proof of the Theorems \ref{Th1real} - \ref{Th2complex} focuses directly on the Hodge dual of a $\G_2$-, $\G_2^*$- or $(\G_2)_{\bC}$-structure. Therefore, after recalling the basic definitions needed in this article and giving a classification of $n$-dimensional Lie algebras with codimension one Abelian ideals in section \ref{preliminaries}, we compute in section \ref{invariants} the values of certain algebraic invariants for the orbits of all Hodge duals in a seven-dimensional real or complex vector space. These algebraic invariants have been used by Westwick \cite{W2} to classify the orbits of three-vectors in seven real dimensions. For the computation we use this classification and a similar classification by Westwick \cite{W1} of the orbits of complex three-vectors up to dimension eight. The concrete values of the mentioned invariants for the Hodge duals and the particular structure of seven-dimensional Lie algebras $\g$ with codimension one Abelian ideals $\uf$ allow us in subsection \ref{firstreduction} to equivalently reformulate our problem in the way that the linear operator $\ad(e_7)|_{\uf}$ for some $e_7\in \g\backslash \uf$ has to be similar to an element in the Lie algebras associated to the stabilizer groups of certain four-forms on $\uf$. This establishes the equivalence of (a)-(c) in Theorem \ref{Th1real}, of (a) and (b) in Theorem \ref{Th1complex} and almost the equivalence of (a) and (b) in Theorem \ref{Th2real} and in Theorem \ref{Th2complex}. Finally, in subsection \ref{secondreduction} we use the well-known results on the structure of complex Jordan normal forms of elements in the symplectic Lie algebras $\mathfrak{sp}(2n,\bR)$ and $\mathfrak{sp}(2n,\bC)$ to reformulate everything totally in properties of the complex Jordan normal form of $\ad(e_7)|_{\uf}$ and finish the proof of Theorem \ref{Th1real}-\ref{Th2complex}.

\section{Preliminaries}\label{preliminaries}

\subsection{$\G_2^{(*)}$-structures on vector spaces} \label{G2vs}
In this subsection, we define $\G_2$-structures, $\G_2^*$-structures and $(\G_2)_{\bC}$-structures on
seven-dimensional real or complex vector spaces $V$. We recall how these structures induce in a natural way a Euclidean metric, a pseudo-Euclidean metric of signature $(3,4)$ or a non-degenerated complex symmetric bilinear form on $V$, respectively and how a Hodge star operator on $V$ can then naturally be defined. Moreover, we remind the reader of the connection of $\G_2$-, $\G_2^*$-structures or $(\G_2)_{\bC}$-structures to the imaginary octonions, imaginary split-octonions or imaginary complex octonions, respectively. We mainly follow \cite{Br} and \cite{CLSS}, to which we also refer for the proofs of all mentioned facts and more details. Note that we adopted the definition of a $\G_2$ and a $\G_2^*$-structure from \cite{CLSS}. We show that it is in accordance with the definition in \cite{Br}.

\begin{definition}
Let $V$ be a seven-dimensional real vector space.
A \emph{$\G_2$-structure} on $V$ is a three-form $\varphi\in \L^3 V^*$
such that there exists a basis $e_1,\ldots,e_7$ of $V$ with
\begin{equation*}
\varphi=e^{127}+e^{347}+e^{567}+e^{135}-e^{146}-e^{236}-e^{245}.
\end{equation*}
Hereby, $e^1,\ldots,e^7\in V^*$ denotes the dual basis of $e_1,\ldots,e_7$.
We call the seven-tuple $(e_1,\ldots e_7)\in V^7$ an \emph{adapted basis}
for the $\G_2$-structure $\varphi$.

A three-form $\tilde{\varphi}\in \L^3 V^*$ is called a
\emph{$\G_2^*$-structure} on $V$ if there exists a basis
$f_1,\ldots,f_7$ of $V$ with 
\begin{equation*}
\tilde{\varphi}=-f^{127}-f^{347}+f^{567}+f^{135}-f^{146}-f^{236}-f^{245}.
\end{equation*}
Again, $f^1,\ldots,f^7\in V^*$ denotes the dual basis of
$f_1,\ldots,f_7$. $(f_1,\ldots,f_7)\in V^7$ is called an
\emph{adapted basis} for the $\G_2^*$-structure $\tilde{\varphi}$.

If we do not want to specify if $\varphi$ is a $\G_2$- or a
$\G_2^*$-structure on $V$ we simply speak of a \emph{$\G_2^{(*)}$-structure}
on $V$.
\end{definition}
\begin{definition}
Let $W$ be a seven-dimensional complex vector space. Then a pair $(\varphi_{\bC},\vol_{\bC})\in \L^3 W^*\times \L^7 W^*$ consisting of a complex volume form $0\neq \vol_{\bC}\in \L^7 W^*$ and a complex three-form $\varphi_{\bC}\in \L^3 W^*$ is called a \emph{$(\G_2)_{\bC}$-structure} if there exists a complex basis $E_1,\ldots,E_7$ of $W$ such that $\vol_{\bC}=E^{1234567}$ and
\begin{equation*}
\varphi_{\bC}=E^{127}+E^{347}+E^{567}+E^{135}-E^{146}-E^{236}-E^{245}.
\end{equation*}
Again $E^1,\ldots,E^7$ denotes the dual basis of $E_1,\ldots,E_7$ and $(E_1,\ldots,E_7)\in W^7$ is called an \emph{adapted basis} for the $(\G_2)_{\bC}$-structure $(\varphi_{\bC},\vol_{\bC})$.
\end{definition}
\begin{remark}\label{re:g2*open}
\begin{itemize}
\item
By definition, the set of all $\G_2$-structures on a seven-dimensional real vector space $V$ and also the set of all $\G_2^*$-structures
on $V$ forms an orbit under the natural action of $\GL(V)$ on $\L^3 V^*$. The stabilizers of these orbits are $\G_2$ and $\G_2^*$, respectively. Thereby, $\G_2$ is the simply-connected compact real form of the complex simply-connected simple Lie group $(\G_2)_{\bC}$ and $\G_2^*$ is the split real form of $(\G_2)_{\bC}$ with fundamental group $\mathbb{Z}_2$. Since $\dim{G_2}=\dim{G_2^*}=14$, these two orbits are open. In fact, they are the only two open orbits under the action of $\GL(V)$ \cite{Hi}. So, in Hitchin's terminology, $\G_2^{(*)}$-structures are \emph{stable} forms \cite{Hi}.
\item
Similarly, the set of all $(\G_2)_{\bC}$-structures $(\varphi_{\bC},vol_{\bC})$ on a complex vector space $W$ forms an orbit in $\Lambda^3 W^*\times \Lambda^7 W^*$ under the natural action of the complex general linear group $\GL(W)$. The stabilizer of $(\varphi_{\bC},vol_{\bC})$ is the stabilizer of $\varphi_{\bC}$ in $\SL(W)$, which is given by the complex simply-connected simple Lie group $(\G_2)_{\bC}$. If we have a three-form $\varphi_{\bC}\in \Lambda^3 W$ for which there exists a basis $E_1,\ldots, E_7$ with $\varphi_{\bC}=E^{127}+E^{347}+E^{567}+E^{135}-E^{146}-E^{236}-E^{245}$, then we can always construct a $(\G_2)_{\bC}$-structure by simply setting $\vol_{\bC}:=E^{1234567}$. Note that this does depend on the chosen complex basis $E_1,\ldots,E_7$. Namely, $F_1:=\frac{1}{\xi}E_1,\ldots, F_7:=\frac{1}{\xi} E_7$ for $\xi\in \bC$, $\xi^3=1$, is also a complex basis such that $\varphi_{\bC}=F^{127}+F^{347}+F^{567}+F^{135}-F^{146}-F^{236}-F^{245}$ giving us $\xi \vol_{\bC}$ as volume form. Note that the set of all such three-forms $\varphi_{\bC}$ forms an open orbit in $\L^3 W^*$ under the natural action of $\GL(W)$.
\item
A $\G_2^{(*)}$-structure $\varphi$ on a real vector space $V$ induces a $(\G_2)_{\bC}$-structure $(\varphi_{\bC},vol_{\bC})$ on the complexification such that $\varphi_{\bC}$ is the complex-linear extension of $\varphi$ and such that $vol_{\bC}:=e^{1234567}$, where $(e_1,\ldots,e_7)$ is an adapted basis for the $\G_2^{(*)}$-structure $\varphi$. Note therefore that if $\varphi$ is a $\G_2$-structure, then $(e_1,\ldots,e_7)$ is also an adapted (complex) basis for $(\varphi_{\bC},vol_{\bC})$ while if $\varphi$ is a $\G_2^*$-structure, then $(i e_1, i e_2,i e_3,i e_4, -e_5, -e_6,  e_7)$ is an adapted basis for $(\varphi_{\bC},vol_{\bC})$. So $\G_2$ and $\G_2^*$-structures lie in the same orbit under the natural action of the complex general linear group.
\end{itemize}
\end{remark}
\begin{lemma}\label{le:metric}
Let $V$ be a seven-dimensional real vector space and $W:=V_{\bC}$ its complexification.
\begin{enumerate}
\item[(a)]
A $\G_2^{(*)}$-structure $\varphi\in \L^3 V^*$ on $V$ induces uniquely a
(pseudo)-Euclidean metric $g$ and a metric volume form $\vol$ on $V$ by
\begin{equation*}
g(v,w) \vol:=\frac{1}{6} (v\hook \varphi) \wedge (w\hook \varphi) \wedge \varphi
\end{equation*}
such that each adapted basis for $\varphi$ is an oriented orthonormal basis.
$g$ is positive definite if $\varphi$ is a $\G_2$-structure and of
signature $(3,4)$ if $\varphi$ is a $\G_2^*$-structure. If $\varphi$
is a $\G_2^*$-structure and $(f_1,\ldots,f_7)$ is an adapted basis, then $g(f_i,f_i)=-1$ for $i=1,2,3,4$ and
$g(f_j,f_j)=1$ for $j=5,6,7$.
\item[(b)]
A $(\G_2)_{\bC}$-structure $(\varphi_{\bC},vol_{\bC})$ induces a non-degenerated symmetric complex bilinear form $g_{\bC}$ on $W$ by
\begin{equation*}
g_{\bC}(v,w) \vol_{\bC}:=\frac{1}{6} (v\hook \varphi_{\bC}) \wedge (w\hook \varphi_{\bC}) \wedge \varphi_{\bC}
\end{equation*}
and a pseudo-Euclidean metric $g_{split}$ of split signature $(7,7)$ on $W_{\bR}$ by $g_{split}:=\Re(g_{\bC})$. If $E_1,\ldots,E_7$ is an adapted basis for $(\varphi_{\bC},vol_{\bC})$, then $g_{\bC}(E_j,E_k)=\delta_{jk}$ for $j,k\in \{1,\ldots,7\}$ and $E_1, iE_1,\ldots,E_7,iE_7$ is an orthonormal basis for $g_{split}$ such that $g_{split}(E_j,E_j)=1$ and $g_{split}(i E_j,i E_j)=-1$ for $j=1,\ldots,7$.
\item[(c)]
Let $\varphi$ be a $\G_2^{(*)}$-structure on $V$, $g$ be the induced pseudo-Euclidean metric on $V$, $vol$ be the induced volume form on $V$ and $(\varphi_{\bC},vol_{\bC})$ be the induced $(\G_2)_{\bC}$-structure on $W$ in the sense of Remark \ref{re:g2*open}. Then $vol_{\bC}$ and the induced complex bilinear form $g_{\bC}$ on $W$ are the complex-linear extensions of $vol$ and $g$, respectively. 
\end{enumerate}
\end{lemma}
\begin{remark}
$\G_2$-structures, $\G_2^*$-structures and $(\G_2)_{\bC}$-structures may be understood through the composition algebra
of the octonions $(\mathbb{O},\langle \cdot,\cdot \rangle)$, of the split-octonions $(\mathbb{O}_s,\langle \cdot,\cdot \rangle_{s})$ or of the complex octonions $(\mathbb{O}_{\bC},\langle \cdot,\cdot \rangle_{\bC})$, respectively.

Therefore, consider the seven-dimensional orthogonal complements $\Im\, \bO$, $\Im\, \bO_{s}$, $\Im\, \bO_{\bC}$
of $\spa{1}\subseteq \bO, \bO_s,\bO_{\bC}$. Then the three-form $\varphi\in \L^3 \Im\, \bO^*$, $\varphi(u,v,w):=\langle u\cdot v,w \rangle$
is a $\G_2$-structure on $\Im\, \bO$, the three-form $\varphi_s\in \L^3 \Im\, \bO_s^*$, $\varphi_s(u,v,w):=\langle u\cdot v,w \rangle_s$ is a $\G_2^*$-structure on $\Im\, \bO_s$ and the three-form $\varphi_{\bC}\in \L^3 \Im\, \bO_{\bC}^*$, $\varphi_{\bC}(u,v,w):=\langle u \cdot v,w \rangle_{\bC}$ together with the volume form $vol_{\bC}:= (v\hook \varphi)\wedge (v\hook \varphi)\wedge \varphi$  is a $(\G_2)_{\bC}$-structure on $\Im\, \bO_{\bC}$, where $v\in \Im\, \bO_{\bC}$ is an arbitrary imaginary complex octonion with $\langle v,v\rangle_{\bC}=1$. Moreover, the (pseudo)-Euclidean metric $g$ on $\Im\, \bO$ or on $\Im\, \bO_s$ induced by $\varphi$ or $\varphi_s$, respectively, is $\langle \cdot,\cdot \rangle$ or $\langle \cdot,\cdot \rangle_s$, respectively. Similarly, the non-degenerated symmetric complex bilinear form $g_{\bC}$ on $\Im\, \bO_{\bC}$ induced by $(\varphi_{\bC},vol_{\bC})$ is exactly $\langle \cdot,\cdot \rangle_{\bC}$. 

To get an adapted basis, consider for a $\G_2^{(*)}$-structure the quaternions $\mathbb{H}$ naturally embedded as a positive definite subspace into $\bO$ and $\bO_s$ and for a $(\G_2)_{\bC}$-structure the complex quaternions $\mathbb{H}_{\bC}$ naturally embedded as a non-degenerated subspace into $\bO_{\bC}$. Choose the standard basis $1,i,j,k$ of $\bH$ or $\bH_{\bC}$ and a normed element $\epsilon\in \bH^{\perp}$ or $\epsilon\in \bH_{\bC}^{\perp}$, respectively. Then, in the case of a $\G_2$ and a $(\G_2)_{\bC}$-structure, $(\epsilon,j\epsilon,i\epsilon,k\epsilon,i,-k,j)$ is an adapted basis while in the case of a $\G_2^*$-structure $(\epsilon,j\epsilon,i\epsilon,k\epsilon,-i,k,j)$ is an adapted basis. Note that other authors \cite{R} call $(i,j,k,\epsilon,i \epsilon, j \epsilon, k \epsilon)$ an adapted basis. If we denote in all three cases this basis by $(F_1,\ldots,F_7)$, then $\varphi$, $\varphi_s$ and $(\varphi_{\bC},vol_{\bC})$ are given by
\begin{equation*}
\begin{split}
\varphi &=F^{123}+F^{145}-F^{167}+F^{246}+F^{257}+F^{347}-F^{356},\quad \varphi_s= F^{123}-F^{145}+F^{167}-F^{246}-F^{257}-F^{347}+F^{356}\\
\varphi_{\bC}&=F^{123}+F^{145}-F^{167}+F^{246}+F^{257}+F^{347}-F^{356},\, vol_{\bC}=F^{1234567}.
\end{split}
\end{equation*}
These are exactly the expressions used in \cite{Br} to define $\G_2^{(*)}$- and $(\G_2)_{\bC}$-structures.
\end{remark}
\begin{definition}\label{def:Hodgedual}
Lemma \ref{le:metric} (a) allows, for a given $\G_2^{(*)}$-structure $\varphi$ on $V$, to define a Hodge star operator
$\star_{\varphi}: \L^* V^*\rightarrow \L^* V^*$ by the usual requirement that for a $k$-form $\psi\in \L^k V^*$ the $(n-k)$-form $\star_{\varphi} \psi\in \L^{n-k} V^*$ is the unique $(n-k)$-form on $V$ such that for all $(n-k)$-forms $\phi\in \L^{n-k} V^*$ we have
\begin{equation*}
\psi\wedge \phi=g(\star_{\varphi} \psi , \phi) \vol.
\end{equation*}
Similarly, Lemma \ref{le:metric} (b) allows us to define for a given $(\G_2)_{\bC}$-structure $(\varphi_{\bC},vol_{\bC})$ on $W$ a Hodge star operator $\star_{\varphi_{\bC}}:\L^* W^* \rightarrow \L^* W^*$ by the requirement that for a complex $k$-form $\psi\in \L^k W^*$ the complex $(n-k)$-form $\star_{\varphi_{\bC}} \psi\in \L^{n-k} W^*$ is the unique $(n-k)$-form such that for all $(n-k)$-forms $\phi\in \L^{n-k} W^*$ we have
\begin{equation*}
\psi\wedge \phi=g_{\bC}(\star_{\varphi_{\bC}} \psi , \phi) \vol_{\bC}.
\end{equation*}
\end{definition}
\begin{remark}\label{re:Hodgedual}
\begin{itemize}
\item
If $(e_1,\ldots,e_7)$ is an adapted basis for a $\G_2^{(*)}$-structure $\varphi$ on $V$, then the Hodge
dual $\star_{\varphi} \varphi$ is given by
\begin{equation}\label{eq:Hodgedual}
\star_{\varphi} \varphi= \epsilon( e^{1256}+e^{3456})+e^{1234}- e^{2467}+e^{2357}+e^{1457}+e^{1367}
\end{equation}
where $\epsilon=1$ if $\varphi$ is a $\G_2$-structure and
$\epsilon=-1$ if $\varphi$ is a $\G_2^*$-structure. Note that the
set of all Hodge duals of a $\G_2$-structure and also the set of
all Hodge duals of a $\G_2^*$-structure form again an open orbit
under the natural action of $\GL(V)$ on $\L^4 V^*$.
\item
If $\Psi\in \L^4 V^*$ is a four-form on the real seven-dimensional vector space $V$ such that there exists a basis $(e_1,\ldots,e_7)$ with 
$\Psi=\epsilon( e^{1256}+e^{3456})+e^{1234}- e^{2467}+e^{2357}+e^{1457}+e^{1367}$, then the stabilizer of $\Psi$ in $\GL^+(V)$ is $\G_2$ for $\epsilon=1$ and $\G_2^*$ for $\epsilon=-1$. Hence such a $\Psi$ together with a volume form $vol$ induces uniquely a $\G_2^{(*)}$-structure $\varphi\in \L^3 V^*$ such that $\star_{\varphi} \varphi=\Psi$. Then $\star_{\varphi} \Psi=\varphi$ and the volume form induced by $\varphi$ is exactly $vol$. Thus, an equivalent description of a $\G_2^{(*)}$-structure can be given as a four-form as above together with an orientation. Although this equivalent description is more appropriate in our case we stick to the standard notation in the literature and only call the three-form $\varphi$ a $\G_2$-structure.
\item
If $(E_1,\ldots,E_7)$ is an adapted basis for a $(\G_2)_{\bC}$-structure $(\varphi_{\bC},vol_{\bC})$ on $W:=V_{\bC}$, then the Hodge
dual $\star_{\varphi_{\bC}} \varphi_{\bC}$ is given by
\begin{equation}\label{eq:Hodgedualcomplex}
\star_{\varphi_{\bC}} \varphi_{\bC}= E^{1234}+E^{1256}+E^{3456}- E^{2467}+E^{2357}+E^{1457}+E^{1367}
\end{equation}
Note that $\star_{\varphi_{\bC}}$ depends also on $\vol_{\bC}$ although we suppressed this dependence in the notation. Note further that if $(\varphi_{\bC},vol_{\bC})$ is induced by a $\G_2^{(*)}$-structure $\varphi$, then $\star_{\varphi_{\bC}}$ is the complex-linear extension of $\star_{\varphi}$. The set of all Hodge duals of a $(\G_2)_{\bC}$-structure forms again an open orbit
under the natural action of $\GL(W)$ on $\L^4 W^*$.
\item
If $\Psi\in \L^4 W^*$ is a four-form on the complex seven-dimensional vector space $W:=V_{\bC}$ such that there exists a basis $(E_1,\ldots,E_7)$ with $\Psi=E^{1234}+E^{1256}+E^{3456}- E^{2467}+E^{2357}+E^{1457}+E^{1367}$, then the stabilizer of $\Psi$ in $\SL(W)$ is exactly $(\G_2)_{\bC}$. Hence 
such a $\Psi$ together with a compatible volume form $vol_{\bC}$ induces uniquely a $(\G_2)_{\bC}$-structure $(\varphi_{\bC},\vol_{\bC})$ such that $\star_{\varphi_{\bC}} \varphi_{\bC}= \Psi$. Then also $\star_{\varphi_{\bC}} \Psi =\varphi_{\bC}$ and as before we see that we may equivalently describe a $(\G_2)_{\bC}$-structure as a pair $(\Psi,vol_{\bC})$ such that there is a basis for which the two forms are in the above standard forms. However, again we stick to the standard notation in the literature and only call the pair $(\varphi_{\bC},vol_{\bC})$ a $(\G_2)_{\bC}$-structure.
\item
Let $(\varphi_{\bC},vol_{\bC})$ be a $(\G_2)_{\bC}$-structure on the complex vector space $W$ and consider the real vector space $W_{\bR}$ with complex structure $i$. Then $\varphi_{\bC}$ is a $(3,0)$-form, $\vol_{\bC}$ a $(7,0)$-form and $g_{split}$ is anti-hermitian with respect to $i$.
\end{itemize}
\end{remark}
\subsection{Cocalibrated $\G_2^{(*)}$-structures on manifolds and Lie algebras}\label{G2mfd}

\emph{$\G_2$-structures} on seven-dimensional
manifolds $M$, i.e. reductions of the frame bundle of $M$ to $\G_2$,
are in one-to-one correspondence to three-forms $\varphi \in \Omega^3 M$
such that $\varphi_p$ is a $\G_2$-structure on $T_p M$ for all $p\in M$ due to the
fact that such a $\varphi$ is pointwise stabilized by $\G_2$.

Similarly, \emph{$\G_2^*$-structures} on seven-dimensional manifolds
$M$ are in one-to-one correspondence to three-forms $\varphi \in \Omega^3 M$
such that $\varphi_p$ is a $\G_2^*$-structure on $T_p M$ for all $p\in M$.

This motivates to call these kinds of three-forms $\varphi\in \Omega^3 M$ from now on$\G_2$- or $\G_2^*$-structures, respectively. 

A $\G_2^{(*)}$-structure $\varphi\in \Omega^3 M$ on a seven-dimensional manifold $M$
induces a Hodge star operator $\star_{\varphi}:\Omega^* M\rightarrow \Omega^* M$. $\varphi\in \Omega^3 M$ is
called \emph{cocalibrated} if the Hodge dual $\star_{\varphi} \varphi$ is
closed.

A \emph{$(\G_2)_{\bC}$-structures} on a $14$-dimensional (real) manifold $M$ is a
reduction of the frame bundle to $(\G_2)_{\bC}\subseteq \GL(7,\bC)\subseteq \GL(14,\bR)$. Note that then
$M$ admits an almost complex structure $J$ coming from the reduction to the subgroup $\GL(7,\bC)$. Hence
$(\G_2)_{\bC}$-structures are in one-to-one correspondence to a triplet
$(J,\varphi_{\bC},\vol_{\bC})\in \Gamma(\End(TM))\times \Omega^3(M,\bC)\times \Omega^7 (M,\bC)$ consisting of an almost complex structure $J$ on $M$, a complex-valued three-form $\varphi_{\bC}$ and a complex volume form $\vol_{\bC}$ such that for all $p\in M$ the pair $((\varphi_{\bC})_p,(\vol_{\bC})_p)\in \L^3 T_p^* M \otimes \bC\times \L^7 T_p^* M\otimes \bC$ is a $(\G_2)_{\bC}$-structure on the complex seven-dimensional space $(T_p M, J_p)$.

In particular, $\varphi_{\bC}\in \Omega^{(3,0)} M$ and $\vol_{\bC}\in \Omega^{(7,0)} M$. Moreover, $(\varphi_{\bC},\vol_{\bC})$ induces a non-degenerated complex bilinear (with respect to $J$) symmetric two-tensor $g_{\bC}\in \Gamma( S^2 T^* M \otimes \bC)$ pointwise as explained above. Note that then $g_{split}:=\Re(g_{\bC})$ is a pseudo-Riemannian metric of split signature $(7,7)$ on $M$ and that $(g_{split},J)$ is an almost anti-hermitian structure on $M$. Moreover, $g_{\bC}$ defines a Hodge star operator $\star_{\varphi_{\bC}}: \Omega^{(k,0)} M\rightarrow \Omega^{(7-k,0)} M$ pointwise as explained above. We say that $(\varphi_{\bC},vol_{\bC})$ is a \emph{cocalibrated $(\G_2)_{\bC}$-structure} if $d \star_{\varphi_{\bC}} \varphi_{\bC}=0$.

A \emph{cocalibrated $\G_2^{(*)}$-structure on a (real) Lie algebra $\g$}
is a $\G_2^{(*)}$-structure on a seven-dimensional real Lie algebra $\g$ with $d \star_{\varphi} \varphi=0$, where one defines an exterior
derivative $d$ on $\L^* \g^*$ in the usual way by identifying $k$-forms
on $\g$ with left-invariant $k$-forms on $G$. That means on one-forms $\alpha\in \L^1 \g^*$ we have $(d\alpha)(X,Y)=-\alpha([X,Y])$ for all one-forms $\alpha\in \L^1 \g^*$ and all $X,Y\in \g$. So cocalibrated, left-invariant
$\G_2^{(*)}$-structures on a Lie group $\G$ are in one-to-one correspondence to cocalibrated
$\G_2^{(*)}$-structures on the associated Lie algebra $\g$ and $d\star_{\varphi} \varphi=0$
is an algebraic equation. Note that, due to Remark \ref{re:Hodgedual}, a Lie algebra admits a cocalibrated $\G_2^{(*)}$-structure if and only if it admits a closed four-form $\Psi$ of the same form as in equation (\ref{eq:Hodgedual}).

Similarly, a \emph{cocalibrated $(\G_2)_{\bC}$-structure on a (complex) Lie algebra $\g$ } is defined as a $(\G_2)_{\bC}$-structure $(\varphi_{\bC},vol_{\bC})$  on the seven-dimensional complex Lie algebra $\g$ with $d \star_{\varphi_{\bC}} \varphi_{\bC}=0$, where the differential $d$ is defined complete analogously. Again we have a one-to-one correspondence between cocalibrated $(\G_2)_{\bC}$-structure on $\g$ and left-invariant cocalibrated $(\G_2)_{\bC}$-structures on a corresponding complex Lie group $G$. Note that the left-invariant almost complex structure $J$ is induced by multiplication with $i$ on $\g$ is integrable. Note further that, again due to Remark \ref{re:Hodgedual}, a Lie algebra admits a cocalibrated $(\G_2)_{\bC}$-structure if and only if it admits a closed four-form $\Psi$ of the same form as in equation (\ref{eq:Hodgedualcomplex}).

\subsection{Lie algebras with codimension one Abelian ideals}\label{classLA}
In this subsection we consider, for $\bF\in \{\bR,\bC\}$, $n$-dimensional $\bF$-Lie algebras with codimension one Abelian ideals. We show how the exterior differential on $k$-forms can be described in an easy way and give a description of all closed $k$-forms on these Lie algebras. Finally we show how one can classify all such Lie algebras.

\begin{lemma}\label{le:derivative}
Let $\g$ be an $n$-dimensional $\bF$-Lie algebra, $\bF\in \{\bR,\bC\}$, with a codimension one Abelian ideal $\uf$. Choose $e_n\in \g\backslash \uf$ and let $e^n\in \Ann{\uf}$ be such that $e^n(e_n)=1$. Then we can canonically identify $\L^k\uf^*$ with $\L^k \Ann{\spa{e_n}}$ as vector spaces using the decomposition $\g=\uf\oplus \spa{e_n}$. Set $f:=\ad(e_7)|_{\uf} \in \mathfrak{gl}(\uf)$. Using the above canonical identification the following statements are true: 
\begin{enumerate}
\item[(i)]
$d\alpha=e^n\wedge(-\alpha\circ f)$ for all $\alpha\in \L^1\uf^*$ and $de^n=0$.
\item[(ii)]
$d\rho=e^n\wedge (f.\rho)$ and $d(e^n\wedge \rho)=0$ for all $\rho\in \L^k \uf^*$. Thereby, the Lie algebra $\mathfrak{gl}(\uf)$ acts in the natural way on $\L^k \uf^*$. 
\item[(iii)]
A $k$-form $\rho\in \L^k \uf^*$ is closed if and only if $f\in L(\GL(\uf)_{\rho})$. Thereby, $\GL(\uf)_{\rho}\subseteq \GL(\uf)$ is the stabilizer group of the $k$-form $\rho$ under the natural action of $\GL(\uf)$ on $\L^k \uf^*$ and $L(\GL(\uf)_{\rho})$ is the associated Lie algebra.
\end{enumerate}
\end{lemma}
\begin{proof}
Let $X,Y\in \g$. Then $de^n(X,Y)=-e^n([X,Y])=0$ since $[X,Y]\in \uf$ and $e^n\in \Ann{\uf}$. This shows $de^n=0$. Next, let $\alpha\in \L^1\Ann{\spa{e_n}}\cong \L^1 \uf^*$. If $X,Y$ are both in $\uf$ we have $[X,Y]=0$ and so
\begin{equation*}
d\alpha(X,Y)=-\alpha([X,Y])=0=(e^n\wedge(-\alpha\circ f))(X,Y).
\end{equation*}
Furthermore, if $Y\in \uf$, then
\begin{equation*}
d\alpha(e_n,Y)=-\alpha([e_n,Y])=-(\alpha\circ f)(Y)=-e^n(e_n)(\alpha \circ f)(Y) =(e^n\wedge(-\alpha\circ f))(e_n,Y).
\end{equation*}
This shows $d\alpha=e^n\wedge (-\alpha\circ f)$ and so (i).

But then $d\rho=e^n\wedge f.\rho$ for all $\rho\in \L^k \uf^*$ follows immediately. Moreover, this also shows 
\begin{equation*}
d(e^n\wedge\rho)=-e^n\wedge d\rho=-e^n\wedge e^n\wedge f.\rho=0
\end{equation*}
for all $\rho\in \L^k \uf^*$ and (ii) follows. 

If $\rho\in \L^k \Ann{e_n}\cong \L^k \uf^*$ is closed (with respect to the differential on $\g$), then $f.\rho=0$. Thus $\exp(tf).\rho=\rho$ and so $\exp(tf)\in \GL(\uf)_{\rho}$ for all $t\in \bR$. Hence $f\in L(\GL(\uf)_{\rho})$. Conversely, if $f\in L(\GL(\uf)_{\rho})$, then by definition $\exp(tf).\rho=\rho$ for all $t\in \bR$. Differentiating at zero gives $f.\rho=0$ and so $d\rho=0$. This finishes the proof of (iii).
\end{proof}
To get a classification of all Lie algebras with a codimension one Abelian ideal we note that the entire structure is encoded in the action of an element $e_n\in \g\backslash \uf$ on the subspace $\uf$ and the action of any other element $e_n'\in\g\backslash \uf$
on $\uf$ is a non-zero multiple of the action of $e_n$. More generally, we have:

\begin{proposition}\label{pro:classLA}
Let $\mathfrak{g}=\bF^{n-1}\rtimes_{\varphi} \bF e_n$ and
$\mathfrak{g}'=\bF^{n-1}\rtimes_{\varphi'} \bF e_n'$ be two
$n$-dimensional $\bF$-Lie algebras, $\bF\in \{\bR,\bC\}$, with Abelian ideal of
codimension $1$. Then $\mathfrak{g}\cong \mathfrak{g}'$ if and only
if there exists $\gamma\in \bF\backslash \{0\}$ such that
$\varphi(e_n)$ and $\gamma \varphi'(e_n')$ are conjugate in $\GL_{n-1}(\bF)$.
Hence $\mathfrak{g}$ is isomorphic to $\mathfrak{g}'$ if and only if there exists $\gamma\in \bF\backslash \{0\}$
such that in the complex Jordan normal forms for $\varphi(e_n)$ and $\gamma \varphi'(e_n')$ for each Jordan block for $\varphi(e_n)$ of size $m$ with $\lambda$ on the diagonal there exists a Jordan block for $\varphi'(e_n')$ of size $m$ with $\gamma \lambda$ on the diagonal.
\end{proposition}
\begin{proof}
"$\Rightarrow$":\\
If $\g$ and so also $\g'$ are Abelian, then the statement is trivial. Next, assume that $\g$ is not Abelian and admits more than one codimension one Abelian ideal. Le $\uf$ be an Abelian ideal in $\g$ of codimension one with $\uf\neq \bF^{n-1}$. Then $V:=\uf\cap \bF^{n-1}$ is a $(n-2)$-dimensional subspace of $\bF^{n-1}$. Since $[\g,\g]\subseteq \uf\cap \bF^{n-1}=V$, we have $\varphi(e_n)(\bF^{n-1})\subseteq V$. Moreover, $\uf\neq \bF^{n-1}$ implies the existence of $\lambda\neq 0$ and $w\in \bF^{n-1}$ such that $u:=w+\lambda e_n\in \uf$. Then, for all $v\in V$, the identities
\begin{equation*}
\varphi(e_n)(v)=[e_n,v]_{\g}=\frac{1}{\lambda} [\lambda e_n,v]_{\g} =\frac{1}{\lambda}[u-w,v]_{\g}=\frac{1}{\lambda}[u,v]_{\g}=0,
\end{equation*}
are true, where the last two identities follow from the fact that $\bF^{n-1}$ and $\uf$ are Abelian. Hence
$\varphi(e_n)|_V=0$ and the Jordan normal form of $\varphi(e_n)$ is
\begin{equation*}
\begin{pmatrix}
0 & 1 & & &\\
 &  0 &  & &\\
 & & 0  & & \\
 & & & \ddots & \\
 & & & & 0
\end{pmatrix}.
\end{equation*}
The same is of course true for $\g'$ and the statement follows for this case (note that then $\g=\h_3\op \bF^{n-3}$ with the three-dimensional Heisenberg algebra $\h_3$).

So we may assume that the unique Abelian ideal of codimension one in $\g$ and $\g'$ is $\bF^{n-1}$. Then each Lie algebra isomorphism $\Phi:\g\rightarrow \g'$ maps $\bF^{n-1}$ isomorphically onto $\bF^{n-1}$ and there has to be $\bF\ni\gamma\neq 0$ and $w\in \bF^{n-1}$ such that $\Phi(e_n)=\gamma e_n'+w$. So, for $\psi:=\Phi|_{\bF^{n-1}}\in \mathfrak{gl}(\bF^{n-1})$ and all $v\in \bF^{n-1}$ we get
\begin{equation*}
(\psi\circ \varphi(e_n))(v)=\Phi([e_n,v]_{\g})=[\Phi(e_n),\Phi(v)]_{\g'}=[\gamma e_n'+w,\psi(v)]_{\g'}=\gamma (\varphi'(e_n')\circ \psi)(v),
\end{equation*}
which implies the statement.

"$\Leftarrow$":\\
By assumption, there exists $\psi\in \GL_{n-1}(\bF)$ and $\gamma\in \bF\backslash\{0\}$ such that
\begin{equation*}
\varphi(e_n)=\psi^{-1}\circ(\gamma \varphi'(e_n'))\circ \psi.
\end{equation*}
Then a short computation shows that
\begin{equation*}
\Psi:\mathfrak{g}\rightarrow \mathfrak{g}',\quad \Psi(v+ \alpha e_n):= \psi(v)+\alpha \gamma e_n',\,\, \forall v\in \bF^{n-1}, \alpha\in \bF                                                             
\end{equation*}
is a Lie algebra isomorphism.
\end{proof}
\begin{remark}\label{re:class}
\begin{itemize}
\item
Note that for $\bF=\bR$ of course not all complex Jordan normal forms are possible for $\varphi(e_n)$. It is well-known that exactly those complex Jordan normal forms are possible where for each complex Jordan block of size $n$ with $\lambda\notin \bR$ on the diagonal we have a complex Jordan block of size $n$ with $\overline{\lambda}$ on the diagonal. Note that this implies that not all complex $n$-dimensional Lie algebras $\g$ with codimension one Abelian ideals are complexifications of real $n$-dimensional Lie algebras with codimension one Abelian ideals.
\item
Proposition \ref{pro:classLA} gives us in fact a classification of all real or complex $n$-dimensional Lie algebras with codimension one Abelian ideals. We may write down a complete list for each dimension by considering step-by-step all possible sizes of the Jordan blocks in the complex Jordan form for $\varphi(e_n)$, choosing the diagonal elements in each Jordan blocks as parameters and restricting these parameters in such a way that they are non-isomorphic for different parameter values but still give all isomorphism classes. This restriction can be carried out in detail in each case using the condition in Proposition \ref{pro:classLA} but it is still a cumbersome job and we will not do it here.
\item
Another way to formulate the essence of Proposition \ref{pro:classLA} is to say that the isomorphism classes of non-Abelian $n$-dimensional $\bF$-Lie algebras, $\bF\in \{\bR,\bC\}$, with codimension one Abelian ideals are in one-to-one correspondence to the orbits of $\PGL_{n-1}(\bF)$ on the projective space $P(\End_{\bF^{n-1}})$. This is a stratified space with the largest strata having codimension $(n-2)$.
\end{itemize}
\end{remark}

\section{Invariants for orbits of $k$-vectors}\label{invariants}

In this section, $\bF$ may be an arbitrary field of characteristic $0$ although we will only need the case $\bF\in \{\bR,\bC\}$ in the following.
Let $V$ be an $n$-dimensional $\bF$-vector space. We define certain
numbers for $k$-vectors $X\in \Lambda^k V$ which are invariant under
the natural action of $\GL(V)$ on $\L^k V$. Some of them were
introduced by Westwick in \cite{W2} to classify all orbits of
three-vectors in seven dimension. Using this classification, the classification of three-vectors up to
eight complex dimensions in \cite{W1} and so-called
\emph{dual isomorphisms} we can determine the values of these invariants for the orbit
of Hodge duals of $\G_2$-structures, $\G_2^*$-structures and $(\G_2)_{\bC}$-structures. For more background on these invariants, 
we refer the reader also to \cite{BG} and \cite{Ca}.

A $k$-vector $X\in \L^k V$ is called \emph{simple} or
\emph{decomposable} if there exist vectors $v_1,\ldots, v_k\in V$
such that $X=v_1\wedge \ldots \wedge v_k$. The cone $G_k(V)\subseteq \L^k V$ of
all decomposable $k$-vectors is called the \emph{Grassmann cone}.

The \emph{(irreducible) length} $l(X)$ of a $k$-vector $X\in \L^k V$
is the smallest non-negative integer $l\in \bN_0$ for which there
exist decomposable $k$-vectors $X_1,\ldots, X_l$ with $X=\sum_{i=1}^l X_i$.
If $U$ is another finite-dimensional $\bF$-vector space, then considering
$X$ as $k$-vector in $V\op U$ does not affect the length of $X$. Moreover, if
$u_1,\ldots u_s$ are linearly independent vectors in $U$, then
$l(X)=l(X\wedge u_1\wedge \ldots \wedge u_s)$.

The \emph{rank} $\rho(X)$ of a $k$-vector $X\in \L^k V$ is
the dimension of the \emph{support} $[X]$ of the $k$-vector $X$, which is defined by
\begin{equation*}
[X]:=\bigcap \{W \textrm{ subspace of $V$} | X\in \L^k W\}.
\end{equation*}
Equivalently, the rank is given by the dimension of the image of the
map $F:V^*\rightarrow \Lambda^{k-1} V$, $F(\alpha):=\alpha\hook X$.

Let $v\in [X]\backslash \{0\}$ and $W$ be a complement of $\spa{v}$
in $[X]$. Then there exists a unique $(k-1)$-vector
$X_1=X_1(v,W)\in \L^{k-1} W$ and a unique $k$-vector
$X_2=X_2(v,W)\in \L^k W$ such that
\begin{equation*}
X=X_1\wedge v+X_2.
\end{equation*}
Let
\begin{equation*}
D(X):=\left\{Y\in \L^{k-1} V \left.| Y=X_1(v,W) \textrm{ for } v\in [X]\backslash \{0\},\, \textrm{$W$ complement of $\spa{v}$ in $[X]$}\right.\right\} 
\end{equation*}
and 
\begin{equation*}
E(X):=\left\{Z\in \L^k V \left.| Z=X_2(v,W) \textrm{ for } v\in [X]\backslash \{0\},\, \textrm{$W$ complement of $\spa{v}$ in $[X]$}\right.\right\}.
\end{equation*}
Set
\begin{equation*}
m(X):=\min\{l(Y)|Y\in D(X)\},\quad r(X):=\min\{l(Z)|Z\in E(X)\}.
\end{equation*}
The quadruple $(\rho(X),l(X),r(X),m(X))$ is invariant under the natural action of $\GL(V)$ on $\L^k V$.

For two-vectors the length is enough to distinguish the orbits under the natural action of the general linear group $\GL(V)$ and the length of a two-vector can easily be computed:
\begin{lemma}\label{le:length2forms}
Let $V$ be an $n$-dimensional $\bF$-vector space. Then:
\begin{enumerate}
\item
$X\in \Lambda^2 V$ is of length $l$ if and only if $X^l\neq 0$ and $X^{l+1}=0$.
\item
$X\in \Lambda^2 V$ has irreducible length $l$ if and only if there are $2l$ linearly independent vectors $v_1,\ldots ,v_{2l}\in V$ such that $X=\sum_{i=1}^l  v_{2i-1}\wedge v_{2i}$.
\end{enumerate}
\end{lemma}
\begin{proof}
(a) is \cite[Theorem 2.11]{BG}.

For (b) note that if $X$ is as in the statement, i.e. $X=\sum_{i=1}^l v_{2i-1}\wedge v_{2i}$ with $v_1,\ldots,v_{2l}\in V$ being linearly independent, then $X^l=l!\, v_1\wedge \ldots\wedge v_{2l}$ and $X^{l+1}=0$. Thus (a) implies that the length is $l$.

If $X$ has length $l$, then, by definition, $X=\sum_{i=1}^l Y_i$ for $Y_i\in G_2(V)$. Hence we may choose vectors $v_j\in V$,
$j=1,\ldots 2l$ such that $Y_i=v_{2i-1}\wedge v_{2i}$. By (a),
\begin{equation*}
v_{1}\wedge\ldots\wedge v_{2l}=Y_1\wedge\ldots\wedge Y_l=\frac{X^l}{l!}\neq 0.
\end{equation*}
Thus $v_1,\ldots, v_{2l}$ are linearly independent and (b) follows.
\end{proof}

In \cite{W2}, Westwick showed that the quadruple $(\rho(X),l(X),r(X),m(X))$ is sufficient to
distinguish between the different orbits of three-vectors in a
seven-dimensional real vector space:
\begin{lemma}\label{le:westwick}
Let $U$ be a seven-dimensional real vector space. If three-vectors
$X\in L^3 U$ and $Y\in \L^3 U$ lie in different orbits under the
natural action of $\GL(U)$ on $\L^3 U$, then
\begin{equation*}
(\rho(X),l(X),r(X),m(X))\neq (\rho(Y),l(Y),r(Y),m(Y)).
\end{equation*}
If $U=V^*$ is the dual space of a seven-dimensional vector space $V$
and $\varphi$ is a $\G_2$-structure on $V$, then
\begin{equation*}
(\rho(\varphi),l(\varphi),r(\varphi),m(\varphi))= (7,5,3,3),
\end{equation*}
and if $\tilde{\varphi}$ is a $\G_2^*$-structure on $V$, then
\begin{equation*}
(\rho(\tilde{\varphi}),l(\tilde{\varphi}), r(\tilde{\varphi}),m(\tilde{\varphi}))= (7,4,2,2).
\end{equation*}
\end{lemma}
In the complex case the different classes of orbits and their lengths and ranks have been determined in \cite{W1}. Using these results and the results obtained in this work before, the values of the two other invariants for the orbit of $(\G_2)_{\bC}$-structures can be computed:
\begin{lemma}\label{le:invariantscomplex}
If $(\vol_{\bC},\varphi_{\bC})$ is a $(\G_2)_{\bC}$-structure on the complex seven-dimensional space $W$, then
\begin{equation*}
(\rho(\varphi_{\bC}),l(\varphi_{\bC}),r(\varphi_{\bC}),m(\varphi_{\bC}))=(7,4,2,2).
\end{equation*}
\end{lemma}
\begin{proof}
\cite{W1} states that $\rho(\varphi_{\bC})=7$ and $l(\varphi_{\bC})=4$. Since $\varphi_{\bC}$ is the complex-linear extension of a $\G_2^*$-structure on a real form of $W$,  Lemma \ref{le:westwick} shows $r(\varphi_{\bC})\leq 2$ and
 $m(\varphi_{\bC})\leq 2$.
 
We first show $r(\varphi_{\bC})=2$. Assume therefore that $r(\varphi_{\bC})\leq 1$. Then we must have a non-zero one-form $\alpha\in W^*$ and a complement $U$ of $\spa{\alpha}$ in $W^*$ such that the unique two-form $\omega\in \L^2 U$ and the unique three-form $\rho\in \L^3 U$ with
$\varphi_{\bC}=\omega\wedge \alpha+\rho$ fulfill $l(\rho)\leq 1$. By Lemma \ref{le:length2forms}, a two-form on a six-dimensional complex vector space has maximal length three. Thus we must have $l(\omega)=3$ and $l(\rho)=1$. By Lemma \ref{le:length2forms} (b), there exists a basis $e^1,\ldots,e^6$ of $U$ such that $\omega=e^{12}+e^{34}+e^{56}$. If $\dim{[\rho]\cap \spa{e^{2i-1},e^{2i}}}=\{2\}$ for some $i\in\{1,2,3\}$, then the length of $\alpha\wedge e^{2i-1}\wedge e^{2i}+\rho$ is one and so the length of $\varphi_{\bC}=\alpha\wedge \omega+\rho$ is at most three, contradicting $l(\varphi_{\bC})=4$. Hence $\dim{[\rho]\cap \spa{e^{2i-1},e^{2i}}}=\{1\}$ for $i=1,2,3$ and we may assume, without loss of generality, that $\rho=e^{246}$. But then $\varphi_{\bC}=e^{246}+e^{12}\wedge \alpha+e^{34}\wedge \alpha+e^{56}\wedge \alpha$ and this means that it is in class $VII$ of \cite{W1}. This class is different from class $X$ in \cite{W1}, which is the equivalence class of a $(\G_2)_{\bC}$-structure. Hence $r(\varphi_{\bC})=2$.

Finally, we show $m(\varphi_{\bC})=2$. Assume therefore that $m(\varphi_{\bC})=1$ (note that by definition $m(\varphi_{\bC})>0$ since $\varphi_{\bC}\neq 0$). Then there exists $0\neq \beta\in V^*$ and a complement $Z$ of $\spa{\beta}$ in $V^*$ such that $\varphi_{\bC}=\beta\wedge \omega+\rho_0$ with $\omega\in \L^2 Z$, $\rho_0\in \L^3 Z$ and $l(\omega)=1$. By Lemma \ref{le:length2forms} (a), $\omega^2=0$. We show that the symmetric complex bilinear form $g_{\bC}$ induced by $(\varphi_{\bC},vol_{\bC})$ is then degenerated, which is the desired contradiction. Therefore, let $v\in \Ann{Z}$ with $\beta(v)=1$. Then
\begin{equation*}
v\hook \varphi_{\bC}=\omega
\end{equation*}
and so
\begin{equation*}
6\, g_{\bC}(v,v)\vol_{\bC}=(v\hook \varphi_{\bC})\wedge (v\hook \varphi_{\bC})\wedge \varphi_{\bC}=\omega^2\wedge \varphi_{\bC} =0.
\end{equation*}
Now let $w\in \Ann{\beta}$ be arbitrary. Then
\begin{equation*}
\begin{split}
6\, g_{\bC}(v,w)\vol_{\bC}&=(v\hook \varphi_{\bC})\wedge (w\hook \varphi_{\bC})\wedge \varphi_{\bC}=\omega\wedge (-\beta\wedge (w\hook \omega)+w\hook \rho_0)\wedge (\beta\wedge \omega+\rho_0)\\
& =-\omega\wedge \beta \wedge (w\hook \omega)\wedge \rho_0+ \omega\wedge (w\hook\rho_0)\wedge \rho_0=0,
\end{split}
\end{equation*}
where the first summand is zero due to $\omega\wedge (w\hook \omega)=\frac{1}{2} w\hook \omega^2=0$ and the second summand is zero since $\omega\wedge (w\hook\rho_0)\wedge \rho_0$ is a seven-vector on the six-dimensional vector space $Z$. But so the symmetric bilinear form $g_{\bC}$ is degenerated. Thus $m(\varphi_{\bC})=2$ as claimed.
 
\end{proof}
We aim at determing the values of these invariants for the Hodge duals of $\G_2^{(*)}$-structures and of $(\G_2)_{\bC}$-structures.
Therefore, we will determining more generally how these invariants transform under Hodge star operators. To deal
with this subject, we introduce the notion of a Grassmann cone preserving maps, see \cite{KPRS}:
\begin{definition}
Let $V_1$, $V_2$ be two finite-dimensional $\bF$-vector spaces and
$g:\Lambda^{k_1} V_1\rightarrow \Lambda^{k_2} V_2$ be a linear map.
We say that $g$ is a \emph{Grassmann cone preserving map} or a
\emph{GCP map} if $g(G_{k_1}(V_1))\subseteq G_{k_2}(V_2)$.
Then $l(g(X))\leq l(X)$ for all $X\in \Lambda^{k_1} V_1$.
$g$ is called a \emph{GCP isomorphism} if it is a vector space isomorphism and
$g$ and $g^{-1}$ are both GCP maps. In this case, $l(g(X))=l(X)$ for all
$X\in \Lambda^{k_1} V_1$.

Each linear map $f:V_1\rightarrow V_2$ induces naturally a GCP map
$f_*:\Lambda^k V_1\rightarrow \Lambda^k V_2$ for all $k\in \bN$.
$f_*$ is a GCP isomorphism if and only if $f$ is a vector space isomorphism.
Such GCP isomorphisms preserve all of the numbers $\rho(X),\, l(X),\,r(X)$
and $m(X)$.

Another important type of GCP isomorphisms is given by so called \emph{dual isomorphism} $\delta$, i.e. by
maps $\delta:\L^k V\rightarrow \L^{n-k} V^*$ such that $\delta(X):=X\hook \vol$ with a volume form $\vol\in \L^n V^*$, $\vol\neq 0$.
\end{definition}
\begin{remark}
\begin{itemize}
\item
If $\delta,\, \tilde{\delta}:\L^k V\rightarrow \L^{n-k} V^*$ are both dual isomorphisms, then $\tilde{\delta}$ is a non-zero multiple of $\delta$.
\item
The standard definition in the literature \cite{KPRS} is to call a linear map $F:\L^k V\rightarrow \L^{n-k} V$, $V$ being an $n$-dimensional $\bF$-vector space, a dual isomorphism if there exists a basis $e_1,\ldots,e_n$ of $V$ such that\\
$F(e_{i_1}\wedge \ldots e_{i_k})=e_{j_1}\wedge \ldots e_{j_{n-k}}$ for all $1\leq i_1<\ldots < i_k \leq n$ and for the uniquely defined $1\leq j_1<\ldots<j_{n-k}\leq n$ with $\{i_1,\ldots,i_k,j_1,\ldots,j_{n-k}\}=\{1,\ldots,n\}$. A relation to our definition can be given by $F= (-1)^{\frac{k(k+1)}{2}} g_*\circ \delta \circ f_*$ with $f:V\rightarrow V$ being the linear map defined by $f(e_i)=(-1)^i e_i$, $\delta$ being the dual isomorphism associated to the volume form $vol=e^{1234567}$ and $g:V^*\rightarrow V$ being the linear map defined by $g(e^i)=e_i$.
\item
Let $\delta$ be a dual isomorphism in our sense for the volume form $vol=e^{1234567}$, $e^1,\ldots, e^7\in V^*$ being a basis of $V^*$. The composition of $\delta$ with the linear map $f_*:\L^{n-k} V^*\rightarrow \L^{n-k} V$ induced by the linear map $f:V^*\rightarrow V$ with $f(e^i)=e_i$ is the Hodge star operator associated to the Euclidean metric and the orientation on $V$ for which $e_1,\ldots, e_n$ is an oriented orthonormal basis. Similarly, Hodge star operators associated to pseudo-Euclidean metrics and a given orientation or to non-degenerated complex symmetric bilinear forms and a compatible complex volume form are compositions of one dual isomorphism and one $GCP$ isomorphism of the type $f_*$.
\end{itemize}
\end{remark}
Hence, to determine the values of the invariants $\rho(X),\, l(X),\, m(X),\, r(X)$ under Hodge star operators it suffices to determine the values of these invariants under dual isomorphisms. Therefore we observe the following:

The image of a decomposable non-zero $k$-vector $X=v_1\wedge \ldots \wedge v_k$ is a
non-zero $(n-k)$-form $\Omega\in \L^{n-k} \Ann{[X]}$, where $\Ann{[X]}$ is the annihilator of $[X]$. Since $\dim{\Ann{[X]}}=n-k$, it has to be decomposable.
Thus $\delta: \L^k V\rightarrow \L^{n-k} V^*$ is a GCP homomorphism. It is, in fact, a GCP isomorphism since the inverse of $\delta(X)=X\hook\vol$ is again a dual isomorphism, namely $\delta^{-1}(\psi)=\psi\hook \nu$ with $\nu\in \L^n V^{**}\cong\L^n V$, $\nu(\vol)=1$. These observations imply
\begin{lemma}\label{le:dualiso}
Let $V$ be an $n$-dimensional $\bF$-vector space, $k\in \{1,\ldots,n-1\}$ and
$\delta:\L^k V\rightarrow \L^{n-k} V^*$ be a dual isomorphism. Then $\delta$ is a GCP isomorphism
and so $l(\delta(X))=l(X)$ for all $X\in \L^k V$. If $r(X)=0$, then $\rho(\delta(X))<n$. Moreover, if $X\in \L^k V$
fulfills $\rho(X)=n$ and $r(X)>0$, then $\rho(\delta(X))=\rho(X)$,
$m(\delta(X))=r(X)$ and $r(\delta(X))=m(X)$.
\end{lemma}
\begin{proof}
We only have to show the second and the third part. Let $0\neq vol\in \L^n V^*$ be the volume form associated to $\delta$. If $X\in \L^k V$ with $r(X)=0$, then there exists $0\neq v\in V$, a complement $W$ of $\spa{v}$ in $V$ and $X_1\in \L^{k-1} V$ such that $X=v\wedge X_1$. But then $\delta(X)=(v\wedge X_1)\hook vol=X_1\hook(v\hook vol)$ and the considerations before the Lemma show $v\hook vol\in \L^{n-1} \spa{v}^0$. Thus $\delta(X)\in \L^{n-k} \spa{v}^0$. Since $\spa{v}^0$ is a subspace of codimension one in $V^*$, we get $\rho(\delta(X))<n$.

Next, let $X$ be a $k$-vector with $\rho(X)=n$ and $r(X)>0$. If the rank of $\delta(X)$ is less than $n$, then there
exists $v\in V$, $v\neq 0$ such that $v\hook\delta(X)= 0$. Thus
\begin{equation*}
\delta(X\wedge v)=(X\wedge v)\hook \vol= v\hook\left(X\hook \vol\right)=v\hook \delta(X)= 0,
\end{equation*}
and so $X\wedge v=0$. We decompose $X=X_1\wedge v+ X_2$ with $X_1\in \L^{k-1} W$, $X_2\in \L^k W$ for some complement $W$ of
$\spa{v}$ in $V$. But then $r(X)>0$ (note that $[X]=V$ due to $\rho(X)=n$) implies $X_2\neq 0$ and so $X\wedge v=X_2\wedge v\neq 0$, a contradiction. Hence $\rho(\delta(X))=n$.

Now we compute the values $r(\delta(X))$ and $m(\delta(X))$. Therefore, let $v\in V$, $v\neq 0$ and $W$ be a complement of $\spa{v}$ in $V$. Decompose
\begin{equation*}
X=X_1\wedge v+X_2
\end{equation*}
uniquely with $X_1\in \L^{k-1} W$ and $X_2\in \L^k W$. Analogously to above, we get $\delta(X_1\wedge v)\in \L^{n-k} \Ann{\spa{v}}$. Moreover, the considerations directly before the lemma imply $\delta(X_2)\in \L^{n-k} \Ann{[X_2]}\subseteq \L^1 \Ann{W}\wedge \L^{n-k-1} \Ann{\spa{v}}$. Thus $\delta(X_2)=\alpha\wedge Y_2$ for some $0\neq \a\in \Ann{W}$ and for $Y_2\in \L^{n-k-1} \Ann{\spa{v}}$. This is true for each possible choice of $v$ and $W$. If we now choose $v$ and $W$ such that $l(X_1)=m(X)$ (note that this is only possible since $\rho(X)=n)$, then $\rho(\delta(X))=n$, the fact that $\Ann{\spa{v}}$ is a complement of $\spa{\a}$ in $V^*$ and that $\delta$ a $GCP$-map imply $m(X)=l(X_1)=l(X_1\wedge v)=l(\delta(X_1\wedge v))\geq r(\delta(X))$. Similarly, if we choose $v,W$ such that $l(X_2)=r(X)$, we obtain $r(X)=l(X_2)=l(\delta(X_2))\geq m(\delta(X))$.

$\delta^{-1}$ is again a dual isomorphism. Thus $r(\delta(X))=0$ would imply $\rho(X)<n$, a contradiction. Hence $r(\delta(X))>0$ and we already proved $\rho(\delta(X))=n$. But then we may apply the just proven and get $m(\delta(X))\geq r(\delta^{-1}(\delta(X)))=r(X)$ and $r(\delta(X))\geq m(\delta^{-1}(\delta(X)))$. Thus equality holds in both cases and the statement is proven.
\end{proof}
Lemma \ref{le:length2forms} and Lemma \ref{le:dualiso} imply the following result on standard forms of $(n-2)$-vectors of length $l$:
\begin{lemma}\label{le:length(n-2)forms}
Let $V$ be an $n$-dimensional $\bF$-vector space and $Y\in \Lambda^{n-2} V$ be an $(n-2)$-vector on $V$. In a wedge product, denote by $\widehat{v}$ for $v\in V$ a vector which is omitted in this product. Then:
\begin{enumerate}
\item
$Y$ has length $l<\frac{n}{2}$ if and only if there exists a basis $w_1,\ldots ,w_n$ of $V$ such that
\begin{equation*}
Y=\sum_{i=1}^l  w_{1}\wedge \ldots \widehat{w_{2i-1}} \wedge \widehat{w_{2i}} \wedge \ldots \wedge w_n.
\end{equation*}
\item
$Y$ has length $l=\frac{n}{2}$ if and only there exists a basis $w_1,\ldots,w_n$ such that
\begin{equation*}
Y=\pm\left(\sum_{i=1}^l  w_{1}\wedge \ldots \widehat{w_{2i-1}} \wedge \widehat{w_{2i}} \wedge \ldots \wedge w_n\right).
\end{equation*}
\end{enumerate}
\end{lemma}
\begin{proof}
Let $\delta:\L^{n-2} V\rightarrow \L^2 V^*$ be a dual isomorphism. Then $Y\in \L^{n-2} V$ is of length $l$ if and only if $\delta(Y)$ is of length $l$
and Lemma \ref{le:length2forms} (b) tells us that this is the case if and only if there exists a basis $v^1,\ldots, v^n\in V^*$ of $V^*$ such that $\delta(Y)=\sum_{i=1}^l v^{2i-1}\wedge v^{2i}$. Now
$\alpha \delta(Y)=Y\hook v^{1\ldots n}$ for some $\alpha\neq 0$ and so $Y\in \L^{n-1} V$ is of length $l$ if and only if there exists a basis $v_1,\ldots, v_n$ such that $Y= \sum_{i=1}^l  \alpha\, v_{1}\wedge \ldots \widehat{v_{2i-1}} \wedge \widehat{v_{2i}} \wedge \ldots \wedge v_n$ holds. Setting $w_j:=|\alpha|^{\frac{1}{n-2}} v_j$ for $j=1,\ldots,n$, we get
\begin{equation*}
Y=\pm (\sum_{i=1}^l  w_{1}\wedge \ldots \widehat{w_{2i-1}} \wedge \widehat{w_{2i}} \wedge \ldots \wedge w_n).
\end{equation*}
This shows (b). If $l<\frac{n}{2}$, the vector $w_n$ appears in each summand. Hence, by changing the sign of $w_n$, if necessary, we can change the overall sign to $+$ and so (a) follows.
\end{proof}
Lemma \ref{le:westwick} and Lemma \ref{le:dualiso} imply
\begin{proposition}\label{pro:inv}
Let $\varphi$ be a $\G_2$-structure on a seven-dimensional real vector
space $V$, let $\tilde{\varphi}$ be a $\G_2^*$-structure on $V$ and let $\varphi_{\bC}$ be a $(\G_2)_{\bC}$-structure
on a seven-dimensional complex vector space $W$. Then
\begin{equation*}
\begin{split}
(\rho(\star_{\vp} \vp),l(\star_{\vp} \vp),m(\star_{\vp} \vp),r(\star_{\vp} \vp))&=(7,5,3,3),\quad (\rho(\star_{\tilde{\vp}} \tilde{\vp}),l(\star_{\tilde{\vp}} \tilde{\vp}),m(\star_{\tilde{\vp}} \tilde{\vp}),r(\star_{\tilde{\vp}} \tilde{\vp}))=(7,4,2,2),\\
(\rho(\star_{\vp_{\bC}} \vp_{\bC}),l(\star_{\vp_{\bC}} \vp_{\bC}),m(\star_{\vp_{\bC}} \vp_{\bC}),r(\star_{\vp_{\bC}} \vp_{\bC}))&=(7,4,2,2).
\end{split}
\end{equation*}
\end{proposition}
So, if $\star_{\varphi}\varphi$ is the Hodge dual of a $\G_2$-structure then, for each choice of a non-zero one-form $\alpha\in V^*$ and of a complement $W$ of $\spa{\alpha}$ in $V^*$ the unique four-form $\Omega\in \L^4 W$ and the unique three-form $\rho\in \L^3 W$ with $\star_{\varphi} \varphi=\Omega+\rho\wedge \alpha$ fulfill $l(\Omega)\geq 3$ and $l(\rho)\geq 3$. Since these lengths are at most three by  \cite{W2} and Lemma \ref{le:length(n-2)forms}, we get equality, i.e. $l(\Omega)=3$ and $l(\rho)=3$. Even more, the values of the other invariants of $\Omega$ and $\rho$ are uniquely determined:
\begin{lemma}\label{le:Jrho}
Let $V$ be a seven-dimensional vector space, let $\varphi\in \L^3 V^*$ be a $\G_2$-structure and let $\star_{\varphi}\varphi$ be its Hodge dual. Choose a non-zero one-form $\alpha\in V^*$ and a complement $W$ of $\spa{\alpha}$ in $V^*$. Denote by $\Omega\in \L^4 W$ and by $\rho\in \L^3 W$ the unique forms such that $\star_{\varphi}\varphi=\Omega+\rho\wedge \alpha$. Then:
\begin{enumerate}
\item
$(\rho(\Omega),l(\Omega),m(\Omega),r(\Omega))=(6,3,2,1)$.
\item
$(\rho(\rho),l(\rho),m(\rho),r(\rho))=(6,3,2,2)$.
\end{enumerate}
\end{lemma}
\begin{proof}
\begin{enumerate}
\item
Since four-forms in five dimensions are of length at most one, $\rho(\Omega)=6$. Let $\beta\in W$, $\beta\neq 0$ and $U$ be a complement of $\spa{\beta}$ in $W$. Let $\tilde{\Omega}\in \L^4 U$, $\tilde{\rho}\in \L^3 U$ be such that $\Omega=\tilde{\Omega}+\tilde{\rho}\wedge \beta$. Again, since four-forms in five-dimensions are of length at most one, $l(\tilde{\Omega})\leq 1$. Moreover, since three-forms in five dimensions are of length at most two by Lemma \ref{le:length(n-2)forms}, $l(\tilde{\rho})\leq 2$. But $l(\Omega)=3$ forces $l(\tilde{\Omega})=1$ and $l(\tilde{\rho})=2$. Thus $r(\Omega)=1$ and $m(\Omega)=2$ as claimed.
\item
Again, since three-forms in five dimensions are of length at most two, $\rho(\rho)=6$. By \cite{W2}, then either $(m(\rho),r(\rho))=(2,2)$ or $(m(\rho),r(\rho))=(1,1)$. Suppose that the second holds. Then there exists $\beta\in W$, $\beta\neq 0$ and a complement $U$ of $\spa{\beta}$ in $W$ such that for the unique three-form $\hat{\rho}\in \L^3 U$ and the unique two-form $\omega\in \L^2 U$ with $\rho=\hat{\rho}+\omega\wedge \beta$ we get
$l(\hat{\rho})=1$. If $\tilde{\Omega}\in \L^4 U$ and $\tilde{\rho}\in \L^3 U$ denote the unique forms such that $\Omega=\tilde{\Omega}+\tilde{\rho}\wedge \beta$, then $l(\tilde{\Omega})\leq 1$ ($U$ is five-dimensional) and the equality
\begin{equation*}
\star_{\varphi}\varphi=\Omega+\rho\wedge \alpha=\tilde{\Omega}+\tilde{\rho}\wedge \beta+\hat{\rho}\wedge \alpha+\omega\wedge \beta\wedge \alpha= (\tilde{\Omega}+\hat{\rho}\wedge \alpha)+(\tilde{\rho}-\omega\wedge \alpha)\wedge \beta
\end{equation*}
is true. Since the length of $\tilde{\Omega}+\hat{\rho}\wedge \alpha$ is at most two, we have a contradiction to $r(\star_{\varphi} \varphi)=3$ (consider the decomposition $V^*=(U\oplus \spa{\alpha})\oplus \spa{\beta}$). Thus $(m(\rho),r(\rho))=(2,2)$ as claimed.
\end{enumerate}
\end{proof}
\begin{remark}
Lemma \ref{le:Jrho} implies obstructions to the existence of cocalibrated $\G_2$-structures on an arbitrary seven-dimensional real Lie algebra. Therefore, let $W\subseteq V^*$, $\alpha\in V^*$,  $\rho\in  \L^3 W$ and $\Omega\in \L^4 W$ be as in Lemma \ref{le:Jrho}.
\begin{itemize}
\item
Let $\delta:\L^4 W\rightarrow \L^2 W^*$ be an arbitrary dual isomorphism. Then Lemma \ref{le:Jrho} (a) and Lemma \ref{le:length2forms} (a) imply that $\delta(\Omega)^3\neq 0$. This can easily be computed by a computer algebra system.
\item
Lemma \ref{le:Jrho} (b) shows that $\rho\in \L^3 W\cong \L^3 W^{**}$ is that kind of three-form on the six-dimensional vector space $W^*$ which induces a complex structure $J_{\rho}$ on $W^*$ as e.g. explained in \cite{CLSS}. Note that there is a quartic invariant $\lambda: \L^3 W\rightarrow (\L^6 W)^{\otimes 2}$ on the six-dimensional vector space $W^*$ which is negative exactly on those three-forms which induce a complex structure on $W^*$. The value of that invariant can also easily be computed by a computer algebra system.
\end{itemize}
\end{remark}
The case of a Hodge dual of a $\G_2^*$- or $(\G_2)_{\bC}$-structure is more complicated. In this paper we will only need
\begin{lemma}\label{le:lengthofOmegaG2star}
\begin{enumerate}
\item
Let $V$ be a seven-dimensional real vector space, let $\Psi$ be the Hodge dual of a $\G_2^*$-structure, let $0\neq \alpha\in V^*$ and $U$ be a complement of $\spa{\alpha}$ in $V^*$. If $\Omega\in \L^4 U$, $\rho\in \L^3 U$ are the unique forms such that $\Psi=\Omega+\rho\wedge \alpha$, then $l(\Omega)=2$ if and only if $g(\alpha,\alpha)=0$ in the induced metric $g$.  
\item
Let $V$ be a seven-dimensional complex vector space, let $\Psi$ be the Hodge dual of a $(\G_2)_{\bC}$-structure, let $0\neq \alpha\in V^*$ and $U$ be a complement of $\spa{\alpha}$ in $V^*$. If $\Omega\in \L^4 U$, $\rho\in \L^3 U$ are the unique forms such that $\Psi=\Omega+\rho\wedge \alpha$, then $l(\Omega)=2$ if and only if $g(\alpha,\alpha)=0$ in the induced non-degenerated complex symmetric bilinear form $g$.  
\end{enumerate}
\end{lemma}
\begin{proof}
\begin{enumerate}
\item
Let $vol\in \L^7 V^*$ be the associated volume form and let $0\neq X\in \L^7 V$ be such that $vol(X)=1$. Set $Z:=\delta(\Psi)=\Psi\hook X$ with the dual isomorphism $\delta:\L^4 V^*\rightarrow \L^3 V$, $\delta(\Phi):=\Phi\hook X$. A short computation in an adapted basis and the corresponding dual basis shows
\begin{equation*}
 g(\beta,\gamma) X=\frac{1}{6} (\beta\hook Z)\wedge (\gamma\hook Z)\wedge Z
\end{equation*}
for all $\beta,\gamma\in V^*$. Set $Z_1:=\delta(\Omega)=\Omega\hook X$, $Z_2:=\delta(\rho\wedge \alpha)=(\rho\wedge \alpha)\hook X$. Then $Z=Z_1+Z_2$, $l(\Omega)=l(\delta(\Omega))=l(Z_1)$, $Z_2\in \L^3 \Ann{\alpha}$ and $Z_1=v\wedge Y_1$ with $0\neq v\in \Ann{U}$, $\alpha(v)=1$ and $Y_1\in \L^2\Ann{\alpha}$. Hence $l(\Omega)=l(Z_1)=l(Y_1)$ and $\alpha\hook Z=\alpha\hook Z_1+\alpha\hook Z_2=Y_1$. This implies the equality
\begin{equation*}
6 g(\alpha,\alpha)X= (\alpha\hook Z_1)\wedge (\alpha\hook Z_1)\wedge (Z_1+Z_2)=Y_1^2\wedge (v\wedge Y_1+Z_2)= v\wedge Y_1^3.
\end{equation*}
But so $g(\alpha,\alpha)=0$ if and only if $Y_1^3=0$, which is by Lemma \ref{le:length2forms} (a) equivalent to $l(\Omega)=l(Y_1)<3$ and by Lemma \ref{pro:inv} equivalent to $l(\Omega)=2$. This proves the statement
\item
Part (b) can be proven in complete analogy to part (a).
\end{enumerate}
\end{proof}
\section{Classification}\label{Class}
In this section we prove Theorem \ref{Th1real} - \ref{Th2complex}. We start in subsection \ref{pro:firstreduction} by showing that the existence problem of a cocalibrated $\G_2$-, $\G_2^*$- or $(\G_2)_{\bC}$-structure on a seven-dimensional $\bF$-Lie algebra $\g$ with codimension one Abelian ideal $\uf$ is equivalent to the existence of certain closed four-forms on $\Ann{\spa{e_7}}$ for $e_7\in \g\backslash \uf$. Lemma \ref{le:derivative} tells us that is useful to determine the Lie algebra of the stabilizer group of such a closed four-form under the natural action of $\GL(\uf)$ on $\L^4 \Ann{\spa{e_7}}\cong \L^4 \uf^*$, which is done afterwards. Altogether, we obtain a rather abstract classification of the Lie algebras in question which admit cocalibrated structures and establish the equivalence of (a)-(c) in Theorem \ref{Th1real}, of (a) and (b) in Theorem \ref{Th1complex} and almost the equivalence of (a) and (b) in Theorem \ref{Th2real} and in Theorem \ref{Th2complex}. To finish the proof of the equivalence of (a) and (b) in Theorem \ref{Th2real} and in Theorem \ref{Th2complex} and to prove the equivalence to the last condition in Theorem \ref{Th1real} - \ref{Th2complex} we use in subsection \ref{secondreduction} well-known results on the structure of the complex Jordan normal forms of $\mathfrak{sp}(2n,\bF)$ to express the existence of cocalibrated structures totally in properties of the complex Jordan normal form of $\ad(e_7)|_{\uf}$. Note that this also makes connection to our classification of seven-dimensional $\bF$-Lie algebras achieved in Proposition \ref{pro:classLA}.

\subsection{First reduction of the problem}\label{firstreduction} 

\begin{proposition}\label{pro:firstreduction}
Let $\g$ be a seven-dimensional $\bF$-Lie algebra with six-dimensional Abelian
ideal $\uf$ and let $e_7\in \g\backslash \uf$.
\begin{enumerate}
\item
If $\bF=\bR$, $\g$ admits a cocalibrated $\G_2$-structure $\varphi\in \L^3 V^*$ if and only if $\g$ admits
a closed four-form $\Omega\in \Lambda^4 \Ann{\spa{e_7}}$ of length $l(\Omega)=3$. This is the case if and only if $\g$ admits a cocalibrated $\G_2^*$-structure such that with respect to the induced pseudo-Euclidean metric the subspace $\uf$ is non-degenerated.
\item
If $\bF=\bR$, $\g$ admits a cocalibrated $\G_2^*$-structure $\varphi\in \L^3 V^*$ if and only if $\g$ admits
a closed four-form $\tilde{\Omega}\in \Lambda^4 \Ann{\spa{e_7}}$ of length $l(\tilde{\Omega})\geq 2$.
\item
Similarly, if $\bF=\bC$, $\g$ admits a cocalibrated $(\G_2)_{\bC}$-structure if and only if there exists a closed four-form $\hat{\Omega}\in \Lambda^4 \Ann{\spa{e_7}}$ of length $l(\hat{\Omega})\geq 2$ and it admits a cocalibrated $(\G_2)_{\bC}$-structure such that $\uf$ is non-degenerated with respect to the induced non-degenerated symmetric bilinear form if and only if a closed four-form $\hat{\Omega}\in \L^4 \Ann{\spa{e_7}}$ with $l(\hat{\Omega})=3$ exists.
\end{enumerate}
\end{proposition}
\begin{proof}
Firstly, let $\Psi$ be the Hodge dual of a cocalibrated $\G_2^{(*)}$-structure or of a cocalibrated $(\G_2)_{\bC}$-structure. Decompose the $\Psi$ into
\begin{equation*}
\Psi= \Omega_1\wedge e^7+\Omega_2
\end{equation*}
with $\Omega_1\in \L^3 \Ann{\spa{e_7}}$, $\Omega_2\in \L^4 \Ann{\spa{e_7}}$, where $e^7$ is the element
in the annihilator $\Ann{\uf}$ of $\uf$ with $e^7(e_7)=1$. Then $d(\Omega_1\wedge e^7)=0$ by Lemma \ref{le:derivative}. Thus $d\Psi=0$ implies
$d\Omega_2=0$. Proposition \ref{pro:inv} and Lemma \ref{le:Jrho} imply $l(\Omega_2)=3$ if $\Psi$ is the Hodge dual of a $\G_2$-structure
and $l(\Omega_2)\geq 2$ if $\Psi$ is the Hodge dual of a $\G_2^*$- or a $(\G_2)_{\bC}$-structure. Moreover, if $\uf$ is a non-degenerated subspace with respect to the induced pseudo-Euclidean metric in the case of the Hodge dual of a $\G_2^*$-structure or with respect to the induced non-degenerated symmetric complex bilinear form in the case of the Hodge dual of a $(\G_2)_{\bC}$-structure then $l(\Omega_2)=3$ by Lemma \ref{le:lengthofOmegaG2star}. Therefore, note that the subspace $\Ann{\uf}=\spa{e^7}\subseteq\g^*$ in the dual space $\g^*$ is non-degenerated if and only if the subspace $\uf\subseteq \g$ in the space $\g$ is non-degenerated. This proves one direction in (a)-(c).

For the other direction in (a) - (c), first assume that $\Omega\in\L^4 \Ann{\spa{e_7}}$ is a closed four-form of length three. By Lemma \ref{le:length(n-2)forms} (and replacing $\Omega$ by $-\Omega$, if necessary) for arbitrary $\epsilon\in\{-1,1\}$ there exists a basis $e^1,\ldots,e^6$
of $\Ann{\spa{e_7}}$ such that
\begin{equation*}
\Omega=\epsilon (e^{1256}+e^{3456})+e^{1234}.
\end{equation*}
Then
\begin{equation*}
\Psi:=\epsilon(e^{1256}+e^{3456})+e^{1234}-e^{2467}+e^{2357}+e^{1457}+e^{1367}
\end{equation*}
is closed (use again Lemma \ref{le:derivative}) and is, for $\bF=\bR$, the Hodge dual of a $\G_2$-structure if $\epsilon=1$ and of a $\G_2^*$-structure
if $\epsilon=-1$ and, for $\bF=\bC$, the Hodge dual of a $(\G_2)_{\bC}$-structure for $\epsilon \in \{-1,1\}$. Moreover, Lemma \ref{le:lengthofOmegaG2star} tells us that $\uf$ is non-degenerated with respect to the induced pseudo-Euclidean metric in the case of the Hodge dual of a $\G_2^*$-structure and with respect to the induced symmetric non-degenerated bilinear form in the case of the Hodge dual of a $(\G_2)_{\bC}$-structure.

Finally, let $\Omega\in \L^4 \Ann{\spa{e_7}}$ be a closed four-form of length two. By Lemma \ref{le:length(n-2)forms} there exists a basis $e^1,\ldots,e^6$
of $\Ann{\spa{e_7}}$ such that
\begin{equation*}
\Omega=e^{1234}+e^{1256}.
\end{equation*}
Let $\Psi$ be the Hodge dual of an arbitrary $\G_2^*$-structure or $(\G_2)_{\bC}$-structure, respectively. Then $r(\Psi)=2$ by Proposition \ref{pro:inv} in both cases and so there exists $f^7\in \g^*$, $f^7\neq 0$ and a complement $W$ of $\spa{f^7}$ in $\g^*$ such that
\begin{equation*}
\Psi=\tilde{\Omega}_1\wedge f^7+\tilde{\Omega}_2
\end{equation*}
with $\tilde{\Omega}_1\in \L^3 W$, $\tilde{\Omega}_2\in \L^4 W$ and the length of $\tilde{\Omega}_2$ is two.
Hence there exists a basis $f^1,\ldots,f^6$ of $W$ such that
\begin{equation*}
\tilde{\Omega}_2=f^{1234}+f^{1256}.
\end{equation*}
Denote by $F\in \GL(\g^*)$ the linear map with $F(f^i)=e^i$ for $i=1,\ldots,7$. Then $F_* \Psi$
is the Hodge dual of a $\G_2^*$-structure or of a $(\G_2)_{\bC}$-structure with adapted basis $e_1,\ldots,e_7$, respectively. Moreover,  
\begin{equation*}
F_*\Psi=F_*(\tilde{\Omega}_1\wedge f^7)+F_*(\tilde{\Omega}_2)=F_*(\tilde{\Omega}_1)\wedge e^7+ e^{1234}+e^{1256}=F_*(\tilde{\Omega}_1)\wedge e^7+\Omega.
\end{equation*}
Thus $F_*\Psi$ is closed (once again Lemma \ref{le:derivative}) and $\g$ admits a cocalibrated $\G_2^*$-structure or a cocalibrated $(\G_2)_{\bC}$-structure, respectively. This finishes the proof.
\end{proof}
Let us note some interesting consequences of Proposition \ref{pro:firstreduction}:
\begin{corollary}\label{cor:G2impliesG2star}
Let $\g$ be a seven-dimensional real Lie algebra with six-dimensional Abelian
ideal $\uf$. Then:
\begin{enumerate}
\item[(a)]
If $\g$ admits a cocalibrated $\G_2$-structure, then it
also admits a cocalibrated $\G_2^*$-structure.
\item[(b)]
$\g$ admits a cocalibrated $\G_2$-structure if and only if $\g$ admits a cocalibrated $\G_2^*$-structure such that the subspace $\uf$ is non-degenerated with respect to the induced pseudo-Euclidean metric.
\item[(c)]
The complexification $\g_{\bC}$ admits a cocalibrated $(\G_2)_{\bC}$-structure if and only if $\g$ admits a cocalibrated $\G_2^*$-structure. Moreover, $\g_{\bC}$ admits a cocalibrated $(\G_2)_{\bC}$-structure such that the subspace $\uf_{\bC}$ is non-degenerated with respect to the induced non-degenerated symmetric bilinear form if and only if $\g$ admits a cocalibrated $\G_2^*$-structure such that the subspace $\uf$ is non-degenerated with respect to the induced pseudo-Euclidean metric and this is the case if and only if $\g$ admits a cocalibrated $\G_2$-structure.
\end{enumerate}
\end{corollary}
By Proposition \ref{pro:firstreduction}, the existence of a cocalibrated $\G_2^{(*)}$-structure or a cocalibrated $(\G_2)_{\bC}$-structure is equivalent to the existence of certain closed four-forms in $\L^4 \Ann{\spa{e_7}}$. By Lemma \ref{le:derivative} (c) such a four-form is closed if and only if the linear map $f:=\ad(e_7)|_{\uf}$ is in the Lie algebra of the stabilizer group under the action of $\GL(\uf)$ of the four-form $\Omega$ considered as an element in $\L^4 \uf^*$. Hence, to proceed, we have to determine the stabilizer groups and the associated Lie algebras of four-forms of length two and three on a six-dimensional $\bF$-vector space.
\begin{lemma}\label{le:stabilizers}
Let $V$ be a six-dimensional $\bF$-vector space, $\bF\in\{\bR,\bC\}$, and $\Omega\in \L^4 V^*$ be a four-form.
\begin{enumerate}
\item[(a)]
Let $l(\Omega)=3$. Then the stabilizer group $\GL(V)_{\Omega}$ of $\Omega$ under the natural action of $\GL(V)$ on $\L^4 V^*$ is given by the set of all symplectic and all anti-symplectic transformations of the symplectic vector space $(V,\omega)$, where $\omega\in \L^2 V^*$ is a two-form with $\frac{1}{2}\omega^2=\Omega$. That means 
\begin{equation*}
\GL(V)_{\Omega}=\{f\in \GL(V)| f^*\omega=\epsilon\omega \textrm{ for some } \epsilon\in \{-1,1\} \}.
\end{equation*}
Its Lie algebra is given by the symplectic Lie algebra $\mathfrak{sp}(V,\omega)$.
\item[(b)]
Let $l(\Omega)=2$. Set $V_4:=\{v\in V| \forall\, w\in V:\ \left((w\wedge v)\hook \Omega\right)^2=0 \}$. Then $V_4$ is a four-dimensional subspace. Moreover, there exists a two-dimensional complementary subspace $V_2$ such that $\Omega=\omega_2\wedge \omega_4$ with $\omega_2\in \L^2 \Ann{V_4}$, $\omega_4\in \L^2 \Ann{V_2}$ and such that $\omega_4$ is of length two. The stabilizer group $\GL(V)_{\Omega}$ of $\Omega$ under the natural action of $\GL(V)$ on $\L^4 V^*$ can be described for $\bF=\bR$ by
\begin{equation*}
\begin{split}
\GL(V)_{\Omega}=\{f\in \GL(V)| & f|_{V_2}=f_2+h,\, f_2\in \GL(V_2), h\in \Hom(V_2,V_4),\\
& f|_{V_4}=\frac{f_4}{\sqrt{|\det(f_2)|}},\, f_4\in \GL(V_4),f_4^*\omega_4=\sgn(\det(f_2))\omega_4\}.
\end{split}
\end{equation*}
and for $\bF=\bC$ by
\begin{equation*}
\begin{split}
\GL(V)_{\Omega} =  \{f\in \GL(V)|  & f|_{V_2}=  f_2+h,\, f_2\in \GL(V_2), h\in \Hom(V_2,V_4), \\
& f|_{V_4}=\frac{f_4}{\lambda},\, f_4\in \Sp(V_4,\omega_4),\lambda\in \bC, \lambda^2=\det(f_2) \}.
\end{split}
\end{equation*}
Thereby, we consider $\omega_4$ as a two-form on $\L^2 V_4^*$ by the canonical identification $V_2^0\cong V_4^*$ induced by the decomposition $V=V_2\oplus V_4$.

Its Lie algebra is in both cases given by
\begin{equation*}
\begin{split}
L(\GL(V)_{\Omega})=\{f\in \mathfrak{gl}(V)| & f|_{V_2}=f_2+h,\, f_2\in \mathfrak{gl}(V_2),h\in \Hom(V_2,V_4),\\
& f|_{V_4}=-\frac{\tr(f_2)}{2}\id_{V_4}+f_4,\,f_4\in \mathfrak{sp} (V_4,\omega_4)\}.
\end{split}
\end{equation*}
\end{enumerate}
\end{lemma}
\begin{proof}
\begin{enumerate}
\item[(a)]
Let $\Omega\in \L^4 V^*$ be of length $3$. By Lemma \ref{le:length2forms} (a) and Lemma  \ref{le:length(n-2)forms} there exists a (linear) symplectic two-form $\omega\in \L^2 V^*$, i.e. a two-form of length three, such that $\Omega=\frac{1}{2}\omega^2$. The two-form $\omega$ provides an isomorphism $V^*\rightarrow V$ and so an isomorphism $\L^2 V^*\rightarrow \L^2 V$. Moreover, $0\neq \frac{\omega^3}{6}\in \L^6 V^*$ provides an isomorphism $\L^2 V\rightarrow \L^4 V^*$ given by $X\mapsto X\hook \frac{\omega^3}{6}$. The natural representation of $\GL(V)$ on $\L^4 V^*$ given by $f.\Psi=(f^{-1})^*\Psi$ gives us a representation on $\L^2 V^*$ using the mentioned isomorphisms. $\GL(V)\ni f$ acts then on $\L^2 V^*$ by $f.\phi:= \det(f)^{-1} (f^t)^*\phi$ for all $\phi\in \L^2 V^*$. Thereby, $f^t$ is the symplectic transpose of $f$, i.e. $f^t$ is defined by $\omega(f(v),w)=\omega(v,f^t(w))$ for all $v,w\in V$. Moreover, the image of $\Omega\in \L^4 V^*$ under all these isomorphisms is exactly $\omega$ as one easily checks in a symplectic basis and its dual basis.

If $\bF=\bR$ we therefore have $f\in \GL(V)_{\Omega}$, i.e. $f.\Omega=\Omega$, if and only $\left(\frac{1}{\sqrt{|\det(f)|}} f^{t}\right)^* \omega=\pm \omega$, where the plus sign appears if $\det(f)>0$ and the minus sign if $\det(f)<0$. Hence $\frac{1}{\sqrt{|\det(f)|}} f^{t}$ is a symplectic transformation of $(V,\omega)$ if $\det(f)>0$ and an anti-symplectic transformation if $\det(f)<0$. Since the determinant of a symplectic transformation is one and the determinant of an anti-symplectic transformation in a six-dimensional space is $-1$, we get $\pm 1=\det\left(\frac{1}{\sqrt{|\det(f)|}} f^{t}\right)=\frac{\det(f)}{|\det(f)|^3}$ and so $\sqrt{|\det(f)|}=1$ in both cases. Thus $f^t$ and so also $f$ is a symplectic transformation or an anti-symplectic transformation. Conversely, due to $\Omega=\frac{1}{2}\omega^2$, it is clear that all symplectic and all anti-symplectic transformations of $(V,\omega)$ stabilize $\Omega$.

If $\bF=\bC$, we get $f\in \GL(V)_{\Omega}$ if and only if $\left(\lambda f^{t}\right)^* \omega=\omega$ for $\lambda\in \bC$ with $\lambda^2=\frac{1}{\det(f)}$. Thus $\lambda f^{t}$ is a symplectic transformation. Hence $1=\det(\lambda f^{t})=\frac{1}{\det(f)^2}$ and so $(f^t)^*\omega=\pm \omega$, i.e. $f^{t}$ and so $f$ is a symplectic or an anti-symplectic transformation of $(V,\omega)$. Note again that it is clear that all symplectic and all anti-symplectic transformations stabilize $\Omega$.

The identity component of the stabilizer group is in both cases exactly the set of all symplectic transformations. Hence the associated Lie algebra is in both cases the symplectic Lie algebra $\mathfrak{sp}(V,\omega)$.
\item[(b)]
Let $\Omega\in \L^4 V^*$ be of length $2$. By Lemma \ref{le:length(n-2)forms} there exists a basis $e_1,\ldots,e_6$ of $V$ such that in the dual basis $\Omega=e^{1234}+e^{1256}$. Then $\Omega=\omega_2\wedge \omega_4$ with $\omega_2=e^{12}$, $\omega_4=e^{34}+e^{56}$ and $\omega_4$ is of length two. Setting $V_2:=\spa{e_1,e_2}$, $V_4:=\spa{e_3,e_4,e_5,e_6}$ we see that $\omega_2\in \L^2 \Ann{V_4}$, $\omega_4\in \L^2 \Ann{V_2}$.

Let $v=\sum_{i=1}^6 \alpha_i e_i$. A short computation shows that $((v\wedge e_1)\hook \Omega)^2=0$ implies $\alpha_2=0$ and that $((v\wedge e_2)\hook \Omega)^2=0$ implies $\alpha_1=0$. Hence $V_4\supseteq \{v\in V| \forall\, w\in V:\, ((v\wedge w)\hook \Omega)^2=0\}$. Another computation shows that for $v\in V_4$ we have in fact $((v\wedge w)\hook \Omega)^2=0$ for any $w\in V$. Thus the identity $V_4=\{v\in V| \forall\, w\in V:\, ((v\wedge w)\hook \Omega)^2=0\}$ is true.

Let $f\in \GL(V)$ such that $f.\Omega=(f^{-1})^*\Omega=\Omega$. Of course, then also $f^*\Omega=\Omega$. This allows us to show that $V_4$ is an invariant subspace for $f$. Therefore, let $v\in V_4$, $w\in V$. We have
\begin{equation*}
\begin{split}
0&=((v\wedge f^{-1}(w))\hook \Omega) \wedge ((v\wedge f^{-1}(w))\hook \Omega)=(f^{-1})^*((v\wedge f^{-1}(w))\hook \Omega) \wedge (f^{-1})^*((v\wedge f^{-1}(w))\hook \Omega).
\end{split}
\end{equation*}
Now $f^*\Omega=\Omega$ implies
\begin{equation*}
\begin{split}
((f^{-1})^*((v\wedge f^{-1}(w))\hook \Omega)(w_2,w_3) & =\Omega(f^{-1}(w),v,f^{-1}(w_2),f^{-1}(w_3))\\
&=\Omega(w,f(v),w_2,w_3)=((f(v)\wedge w)\hook \Omega)(w_2,w_3)
\end{split}
\end{equation*}
for all $w_2,w_3\in V$ and so $(f^{-1})^*((v\wedge f^{-1}(w))\hook \Omega)=(f(v)\wedge w)\hook \Omega$. Thus also $((f(v)\wedge w)\hook \Omega)^2=0$.
Hence $f(v)\in V_4$ and $V_4$ is an invariant subspace for $f$. Set $g_4:=f|_{V_4}\in \GL(V_4)$ and define $f_2\in \GL(V_2)$ and $h\in \Hom(V_2,V_4)$ by the equation $f|_{V_2}=f_2+h$. Then $f^{-1}|_{V_4}=g_4^{-1}$ and $f^{-1}|_{V_2}=f_2^{-1}-g_4^{-1}\circ h\circ f_2^{-1}$. Let $v_1,v_2\in V_2$ be such that $\omega_2(v_1,v_2)=1$. For $v_3,v_4\in V_4$ we get
\begin{equation*}
\begin{split}
(g_4^*\omega_4)(v_3,v_4)&=\omega_2(v_1,v_2)\cdot \omega_4(g_4(v_3),g_4(v_4))=\Omega(v_1,v_2,g_4(v_3),g_4(v_4))=\Omega(v_1,v_2,f(v_3),f(v_4))\\
&=( (f^{-1})^*\Omega)(v_1,v_2,f(v_3),f(v_4))=\Omega (f^{-1}(v_1),f^{-1}(v_2),v_3,v_4)\\
&=\omega_2(f^{-1}(v_1),f^{-1}(v_2))\cdot \omega_4(v_3,v_4)=\omega_2( f_2^{-1}(v_1),f_2^{-1}(v_2))\cdot \omega_4(v_3,v_4)\\
&=\frac{1}{\det(f_2)} \omega_4(v_3,v_4).
\end{split}
\end{equation*}
Hence $\det(f_2) g_4^*\omega_4=\omega_4$.

Let $\bF=\bR$. If $\det(f_2)>0$ we have $f_4:=\sqrt{\det(f_2)} g_4\in \Sp(V_4,\omega_4)$ whereas for $\det(f_2)<0$ we get that $f_4:=\sqrt{|\det(f_2)|}g_4$ is an anti-symplectic transformation of $(V_4,\omega_4)$. This shows that the stabilizer group $\GL(V)_{\Omega}$ is contained in 
\begin{equation*}
\begin{split}
G:=\{f\in \GL(V)| & f|_{V_2}=f_2+h,\, f_2\in \GL(V_2), h\in \Hom(V_2,V_4),\\
& f|_{V_4}=\frac{f_4}{\sqrt{|\det(f_2)|}},\, f_4\in \GL(V_4),f_4^*\omega_4=\sgn(\det(f_2))\omega_4\}.
\end{split}
\end{equation*}
A direct calculation shows that each element of $G$ stabilizes $\Omega$. Thus $G=\GL(V)_{\Omega}$ as claimed. The elements of $\GL(V)_{\Omega}$ which lie in the identity component $(\GL(V)_{\Omega})_0$ are characterized by the properties $\det(f_2)>0$ and $f_4\in \Sp(V_4,\omega_4)$. Hence it follows that the associated Lie algebra $L(\GL(V)_{\Omega})$ has the claimed form.

Finally, consider $\bF=\bC$. Then $\det(f_2) g_4^*\omega_4=\omega_4$ states that $f_4:=\lambda g_4\in \Sp(V_4,\omega_4)$ for all $\lambda\in \bC$ with $\lambda^2=\det(f_2)$. This shows that $\GL(V)_{\Omega}$ is contained in
\begin{equation*}
\begin{split}
H:=\{f\in \GL(V)| & f|_{V_2}=f_2+h,\, f_2\in \GL(V_2), h\in \Hom(V_2,V_4),\\
& f|_{V_4}=\frac{f_4}{\lambda},\, f_4\in \Sp(V_4,\omega_4),\lambda\in \bC, \lambda^2=\det(f_2) \}.
\end{split}
\end{equation*}
Conversely, a simple calculation shows that each element of $H$ stabilizes $\Omega$. Thus $\GL(V)_{\Omega}=H$ as claimed. Similarly to the real case we see that $L(\GL(V)_{\Omega})$ has the stated form.
\end{enumerate}
\end{proof}
Lemma \ref{le:stabilizers} and Proposition \ref{pro:firstreduction} imply the following two theorems:
\begin{theorem}\label{th:abstractclassificationreal}
Let $\g$ be a seven-dimensional real Lie algebra with codimension one Abelian ideal $\uf$. Let $e_7\in \g\backslash \uf$, $\omega_6\in \L^2 \uf^*$ be a non-degenerated two-form on $\uf$, $V_4$ be a four-dimensional subspace of $\uf$, $V_2$ be a complementary two-dimensional subspace of $V_4$ in $\uf$ and $\omega_4\in \L^2 V_4^*$ be a non-degenerated two-form on $V_4$. Then:
\begin{enumerate}
\item
$\g$ admits a cocalibrated $\G_2$-structure if and only if $\g$ admits a cocalibrated $\G_2^*$-structure such that $\uf$ is a non-degenerated subspace of $\g$ with respect to the induced pseudo-Euclidean metric on $\g$ and this is the case if and only if $\ad(e_7)|_{\uf}\in \mathfrak{gl}(\uf)$ is similar under the action of $\GL(\uf)$ to an element in $\mathfrak{sp}(\uf,\omega_6)$.
\item
$\g$ admits a cocalibrated $\G_2^*$-structure if and only if $\ad(e_7)|_{\uf}\in \mathfrak{gl}(\uf)$ is similar under the action of $\GL(\uf)$ to an element in $\mathfrak{sp}(\uf,\omega_6)$ or to an element in
\begin{equation*}
\{f\in \mathfrak{gl}(V)| f|_{V_2}=f_2+h,\, f_2\in \mathfrak{gl}(V_2),h\in \Hom(V_2,V_4),\,\, f|_{V_4}=-\frac{\tr(f_2)}{2}\id_{V_4}+f_4,\,f_4\in \mathfrak{sp} (V_4,\omega_4)\}.
\end{equation*}
\end{enumerate}
\end{theorem}
\begin{theorem}\label{th:abstractclassificationcomplex}
Let $\g$ be a seven-dimensional complex Lie algebra with codimension one Abelian ideal $\uf$. Let $e_7\in \g\backslash \uf$, $\omega_6\in \L^2 \uf^*$ be a non-degenerated two-form on $\uf$, $V_4$ be a four-dimensional subspace of $\uf$, $V_2$ be a complementary two-dimensional subspace of $V_4$ in $\uf$ and $\omega_4\in \L^2 V_4^*$ be a non-degenerated two-form on $V_4$. Then:
\begin{enumerate}
\item
$\g$ admits a cocalibrated $(\G_2)_{\bC}$-structure such that $\uf$ is a non-degenerated subspace of $\g$ with respect to the induced symmetric complex bilinear form on $\g$ if and only if $\ad(e_7)|_{\uf}\in \mathfrak{gl}(\uf)$ is similar under the action of $\GL(\uf)$ to an element in $\mathfrak{sp}(\uf,\omega_6)$.
\item
$\g$ admits a cocalibrated $(\G_2)_{\bC}$-structure if and only if $\ad(e_7)|_{\uf}\in \mathfrak{gl}(\uf)$ is similar under the action of $\GL(\uf)$ to an element in $\mathfrak{sp}(\uf,\omega_6)$ or to an element in
\begin{equation*}
\{f\in \mathfrak{gl}(V)| f|_{V_2}=f_2+h,\, f_2\in \mathfrak{gl}(V_2),h\in \Hom(V_2,V_4),\,\, f|_{V_4}=-\frac{\tr(f_2)}{2}\id_{V_4}+f_4,\,f_4\in \mathfrak{sp} (V_4,\omega_4)\}.
\end{equation*}
\end{enumerate}
\end{theorem}
\subsection{Second reduction of the problem}\label{secondreduction}
Theorem \ref{th:abstractclassificationreal} (a) states the equivalence of (a)-(c) in Theorem \ref{Th1real} and Theorem \ref{th:abstractclassificationcomplex} (a) states the equivalence of (a) and (b) in Theorem \ref{Th1complex}. Theorem \ref{th:abstractclassificationreal} (b) shows that condition (b) in Theorem \ref{Th2real} implies condition (a) in the same Theorem and Theorem \ref{th:abstractclassificationcomplex} (b) shows that condition (b) in Theorem \ref{Th2complex} implies condition (a) in the same Theorem. Moreover, we get from these theorems that it suffices to show that each element in $\mathfrak{gl}(\uf)$, which is similar under $\GL(\uf)$ to an element in $\mathfrak{sp}(\uf,\omega_6)$, is also similar under $\GL(\uf)$ to an element in
\begin{equation*}
\{f\in \mathfrak{gl}(V)| f|_{V_2}=f_2+h,\, f\in \mathfrak{gl}(V_2),h\in \Hom(V_2,V_4),\,\, f|_{V_4}=-\frac{\tr(f_2)}{2}\id_{V_4}+f_4,\,f_4\in \mathfrak{sp} (V_4,\omega_4)\}
\end{equation*}
to get that condition (a) implies condition (b) in Theorem \ref{Th2real} and Theorem \ref{Th2complex}, respectively. 
Therefore, it is obviously useful to know something about the structure of complex Jordan normal forms of elements in $\mathfrak{sp}(2n,\bF)$, at least for $n=2$ and $n=3$. 

Moreover, it remains to prove the equivalence of the last condition in Theorem \ref{Th1real} -  \ref{Th2complex} to one of the previous ones. This last condition is also expressed in properties of the complex Jordan normal form of $\ad(e_7)|_{\uf}$. This should be enough motivation to recall
the following well-known results, see e.g \cite[Theorem 2.7]{MMRR} or also \cite[Theorem 2.4]{MX} for $\bF=\bR$, on the complex Jordan normal forms of elements in $\mathfrak{sp}(2n,\bF)$:
\begin{proposition}\label{pro:JNF}
Let $(V,\omega)$ be a symplectic vector space over $\bF\in \{\bR,\bC\}$. Then a linear transformation $f\in \GL(V)$ is similar under the action of $\GL(V)$ to an element in $\mathfrak{sp}(V,\omega)$ if and only if the complex Jordan normal form of $f$ has the property that for all $m\in \mathbb{N}$ and all $0\neq \lambda$ the number of Jordan blocks of size $m$ with $\lambda$ on the diagonal equals the number of Jordan blocks of size $m$ with $-\lambda$ on the diagonal and the number of Jordan blocks of size $2m-1$ with $0$ on the diagonal is even.
\end{proposition}
\begin{proof}[Proof of Theorem \ref{Th1real} - \ref{Th2complex}]
Theorem \ref{th:abstractclassificationreal}, Theorem \ref{th:abstractclassificationcomplex} and Proposition \ref{pro:JNF} imply Theorem \ref{Th1real} and Theorem \ref{Th1complex}.

We remarked before Proposition \ref{pro:JNF} what is left to prove Theorem \ref{Th2real} and Theorem \ref{Th2complex}. But now, since conditions (c) and (d) in Theorem \ref{Th1real} and conditions (b) and (c) in Theorem \ref{Th1complex} are equivalent, respectively, we may and will proceed as follows to finish the proof of Theorem \ref{Th2real} and Theorem \ref{Th2complex}:
\begin{itemize}
\item
First step: Show that condition (d) in Theorem \ref{Th1real} implies condition (c) in Theorem \ref{Th2real} and, similarly, that condition (c) in Theorem \ref{Th1complex} implies condition (c) in Theorem \ref{Th2complex}.
\item
Second step: Show that the conditions (c) and (b) in Theorem \ref{Th2real} are equivalent and, similarly, that the conditions (c) and (b) in Theorem \ref{Th2complex} are equivalent.
\end{itemize}

\emph{First step}:

Let $A\in \mathbb{C}^{6\times 6}$ be a matrix in complex Jordan normal form such that for all $m\in \mathbb{N}$ and all $0\neq \lambda\in \bC$ the number of Jordan blocks of size $m$ with $\lambda$ on the diagonal is the same as the number of Jordan blocks of size $m$ with $-\lambda$ on the diagonal and the number of Jordan blocks of size $2m-1$ with $0$ on the diagonal is even.

Number consecutively the diagonal elements of the complex Jordan normal form by $\lambda_1,\ldots,\lambda_6$ and the Jordan blocks of the complex Jordan normal form by $1,\ldots, m$, both from the upper left to the lower right. Let $\JB(i)$ for all $i=1,\ldots, 6$ be the number of the Jordan block in which the corresponding generalized eigenvector lies. The assumptions on $A$ imply that we can portion $\{1,\ldots,6\}$ in the following way into three subsets $I_1,\,I_2,\,I_3$ of cardinality two:
\begin{itemize}
\item
We can group the Jordan blocks with non-zero diagonal elements into pairs of Jordan blocks of the same size with $\lambda$ and $-\lambda$, $\lambda\neq 0$ on the diagonal. Construct now subsets $I_1,\ldots,I_r$ of cardinality two by going successively through all these pairs of Jordan blocks and putting successively the two indices corresponding to the first,$\ldots$, $l$-th, $\ldots$ diagonal element in the two Jordan blocks in one $I_k$. By the index $i$ corresponding to the $l$-th diagonal element in a certain Jordan block we mean that $i\in \{1,\ldots,6\}$ such the $i$-th diagonal element of the big matrix $A$ is the the $l$-th diagonal element in the Jordan block.
\item
Similarly, we can group the Jordan blocks with zero on the diagonal and of odd size into pairs of the same size and construct subsets $I_{r+1},\ldots, I_s$ taking successively all these pairs of Jordan blocks and putting again the two indices corresponding to the first,$\ldots$, $l$-th, $\ldots$ diagonal element in the two Jordan block in one $I_k$.
\item
Finally, we construct subsets $I_{s+1},\ldots,I_3$ by taking successively the Jordan blocks with $0$ in the diagonal of even size and putting together the two indices corresponding to the $(2l-1)$-th and $2l$-th diagonal element. 
\end{itemize}
By construction, $\sum_{i\in I_k} \lambda_i=0$ for all $k=1,2,3$ and so condition (i) in Theorem \ref{Th2real} (c) and Theorem \ref{Th2complex} (c) is fulfilled, respectively. Moreover, if $i_1\in I_1$, $i_2\in I_2$ are such that $\JB(i_1)=\JB(i_2)$, then by construction also $\JB(j_1)=\JB(j_2)$ for the unique $j_k\in I_k$ such that $I_k=\{i_k,j_k\}$ for $k=1,2$. This show that also condition (ii) in Theorem \ref{Th2real} (c) and Theorem \ref{Th2complex} (c) is fulfilled, respectively. Finally, we prove that also condition (iii) in Theorem \ref{Th2real} (c) and Theorem \ref{Th2complex} (c) is fulfilled, respectively. By symmetry we may assume that there is $i_2\in I_2$ such that $\JB(i_1)=\JB(j_1)=\JB(i_2)$, $\{i_1,j_1\}=I_1$. Then $\lambda_{i_1}+\lambda_{j_1}=0$ and $\lambda_{i_1}=\lambda_{j_1}=\lambda_{i_2}$ imply $0=\lambda_{i_1}=\lambda_{j_1}=\lambda_{i_2}$. By construction, $\JB(j_1)=\JB(j_2)=\JB(i_2)=\JB(i_1)$ and so $\lambda_{j_2}=0$ for $j_2\in I_2$, $j_2\neq i_2$. This finishes the proof of the first part.

\emph{Second step}:

For this part of the proof, note that we follow that standard convention on the form of Jordan blocks which puts the $1$s on the superdiagonal.

We first show that condition (b) implies condition (c) in Theorem \ref{Th2real} and in Theorem \ref{Th2complex}, respectively. Let $f:=\ad(e_7)|_{\uf}$, $e_7\in \g\backslash \uf$. We may assume that we have a four-dimensional invariant subspace $V_4\subseteq \uf$ and a two-dimensional complementary subspace $V_2\subseteq \uf$ such that $f|_{V_2}=f_2+h$, $f_2\in \mathfrak{gl}(V_2)$, $h\in \Hom(V_2,V_4)$ and $f|_{V_4}=f_4-\frac{\tr(f_2)}{2} \id_{V_4}$ with $f_4\in \mathfrak{sp}(V_4,\omega)$ for some non-degenerated two-form $\omega$ on $V_4$. Choose a Jordan basis $v_1,\ldots,v_4$ of $f_4$ and denote by $\mu_1,\ldots,\mu_4$ the corresponding diagonal elements. To simplify notation we will say in the following that certain vectors $u_1,\ldots,u_s$ are a Jordan basis of a linear map if they there is a permutation making them into a Jordan basis. Then Proposition \ref{pro:JNF} tells us that we may assume $\mu_1=-\mu_2$ and $\mu_3=-\mu_4$. Set $\lambda_i:=\mu_i-\frac{\tr(f_2)}{2}$. The vectors $v_1,\ldots, v_4$ are also a Jordan basis of $f|_{V_4}$ such that $v_i$ and $v_j$ are in one Jordan block for $f_4$ with $\mu_i$ on the diagonal if and only if $v_i$ and $v_j$ are in one Jordan block for $f|_{V_4}$ with $\lambda_i$ on the diagonal. By \cite[Theorem 4.1.4]{GLR} there is a Jordan basis $w_1,\ldots, w_6$ of $f$ such that for all $i,j\in \{1,\ldots,4\}$ the vectors $v_i$ and $v_j$ are in the same Jordan block for $f|_{V_4}$ with $\lambda_i$ on the diagonal if and only if $w_i$ and $w_j$ are in the same Jordan block for $f$ with $\lambda_i$ on the diagonal. Since the characteristic polynomial of $f$ is the product of the characteristic polynomials of $f|_{V_4}$ and $f_2$, the Jordan base vectors $w_5$ or $w_6$ are in Jordan blocks with $\lambda_5$ or $\lambda_6$ on
the diagonal, respectively, where $\lambda_5$, $\lambda_6$ are the roots of the characteristic polynomial of $f_2$. In particular, $\tr(f_2)=\lambda_5+\lambda_6$. This allows us now to prove that condition (i)-(iii) in Theorem \ref{Th2real} (c) and Theorem \ref{Th2complex} (c) is fulfilled, respectively.
\begin{itemize}
\item
We get
\begin{equation*}
\lambda_1+\lambda_2=\mu_1+\mu_2-\tr(f_2)=-\lambda_5-\lambda_6,\quad \lambda_3+\lambda_4=\mu_3+\mu_4-\tr(f_2)=-\lambda_5-\lambda_6,
\end{equation*}
which is exactly condition (i).
\item
If $w_{i_1}$ and $w_{i_2}$ are in one Jordan block for $f$ with $\lambda_{i_1}$ on the diagonal for $i_1\in \{1,2\}$, $i_2\in \{3,4\}$, then $v_{i_1}$ and $v_{i_2}$ are in one Jordan block for $f_4$ with $\mu_{i_1}$ on the diagonal. We may have $\mu_{i_1}=\mu_{i_2}=0$. If this is not the case, Proposition \ref{pro:JNF} implies that $f_4$ has to contain two Jordan blocks of size two, one with $\mu_{i_1}$ and the other with $-\mu_{i_1}$ on the diagonal and so $v_{j_1}$, $v_{j_2}$ are in one Jordan block for those $j_1$, $j_2$ with $\{i_1,j_1\}= \{1,2\}$, $\{i_2,j_2\}= \{3,4\}$. Hence $\lambda_{i_1}=\lambda_{i_2}=-\frac{\lambda_5+\lambda_6}{2}$ or also $w_{j_1}$, $w_{j_2}$ are in one Jordan block. This is condition (ii).
\item
If $w_1$, $w_2$ and $w_{i_2}$ for some $i_2\in \{3,4\}$ or $w_{i_1}$, $w_3$ and $w_4$ for some $i_1\in \{1,2\}$ are in one Jordan block for $f$
with $\lambda$ on the diagonal, then $v_1$, $v_2$ and $v_{i_2}$ or $v_{i_1}$, $v_3$ and $v_4$ are in one Jordan block for $f_4$
with $\lambda+\frac{\lambda_5+\lambda_6}{2}$ on the diagonal. But then Proposition \ref{pro:JNF} tells us that $v_1$, $v_2$, $v_3$ and $v_4$ are in one Jordan block for $f_4$ with $0$ on the diagonal. Hence $w_1$, $w_2$, $w_3$, $w_4$ are in one Jordan block for $f$ with $-\frac{\lambda_5+\lambda_6}{2}$ on the diagonal. This is condition (iii).
\end{itemize}

Finally we show that condition (c) implies condition (b), both in Theorem \ref{Th2real} as in Theorem \ref{Th2complex}. Let $A\in\bC^{6\times 6}$ be in complex Jordan normal form and assume that it fulfills all the conditions in Theorem \ref{Th2real} (c) or in Theorem \ref{Th2complex} (c), respectively. We number consecutively the diagonal elements of the complex Jordan normal form by $\lambda_1,\ldots,\lambda_6$ and the Jordan blocks of the complex Jordan normal by $1,\ldots, m$, both from the upper left to the lower right and denote by $\JB(i)$ for all $i=1,\ldots, 6$, the number of the Jordan block in which the corresponding generalized eigenvector lies. Let $I_1$, $I_2$ and $I_3$ be a partition of $\{1,\ldots, 6\}$ as in condition (c). We may assume that $\JB(i_k)=\JB(i_3)$ for $i_k\in I_k$, $k=1,2$, $i_3\in I_3$ implies $i_k<i_3$ simply by redefining $I_k$ and $I_3$ if this is not the case (note therefore that $\lambda_{i_k}=\lambda_{i_3}$). Set $V_2:=\spa{e_i|i\in I_3}$ and $V_4:=\spa{e_j|j\in I_1\cup I_2}$. Since $\JB(i_k)=\JB(i_3)$ for $i_k\in I_k$, $k=1,2$, $i_3\in I_3$ implies $i_k<i_3$, $V_4$ is an invariant subspace for $A$. That means there are $A_2\in \mathfrak{gl}(V_2)$, $H\in \Hom(V_2,V_4)$ and $A_4\in \mathfrak{gl}(V_4)$ such that $A|_{V_2}=A_2+H$ and $A|_{V_4}=A_4$. Moreover, $A_4$ is in complex Jordan normal form and so $B:=A|_{V_4}+\frac{\tr(A_2)}{2}I_4$ is also in complex Jordan normal form. We claim that $B\in \mathfrak{sp}(4,\bF)$. We have to check that $B$ fulfills all the conditions in Proposition \ref{pro:JNF}. We use the conditions (i)-(iii) Theorem \ref{Th2real} (c) and in Theorem \ref{Th2complex} (c), which give us information on the structure of $A$, and discuss what this implies for the structure of $B$. First, note that $\tr(A_2)=\sum_{i\in I_3}\lambda_i$ implies that the diagonal elements of $B$ are given by $\mu_j=\lambda_j+\frac{\sum_{i\in I_3} \lambda_i}{2}$, $j\in  I_1\cup I_2$. Then:
\begin{itemize}
\item
Condition (i) states that if there is a diagonal element of $B$ with value $\mu=\lambda+\frac{\sum_{i\in I_3} \lambda_i}{2}$, then there is always a different diagonal element of $B$ with value $-\sum_{i\in I_3} \lambda_i-\lambda+\frac{\sum_{i\in I_3} \lambda_i}{2}=-\left(\lambda+\frac{\sum_{i\in I_3} \lambda_i}{2}\right)=-\mu$.
\item
Condition (ii) implies that if $B$ contains a Jordan block of size $2$ with $\mu=\lambda+\frac{\sum_{i\in I_3} \lambda_i}{2}$ on the diagonal, then $\lambda=-\frac{\sum_{i\in I_3} \lambda_i}{2}$, i.e. $\mu=0$, or it also contains another Jordan block of size $2$ with $-\sum_{i\in I_3} \lambda_i-\lambda+\frac{\sum_{i\in I_3} \lambda_i}{2}=-\mu$ on the diagonal.
\item
Condition (iii) states that there cannot be any Jordan block of size $3$ in $B$ and there can only be a Jordan block of size $4$ in $B$ if the diagonal elements are equal to $0$.
\end{itemize}
If there is a Jordan block of size $1$ in $B$ with $\mu$ on the diagonal, then by condition (i) there has to be another Jordan block with $-\mu$ on the diagonal. If this Jordan block would have size more than one, then condition (iii) implies that it has to have size two and condition (ii) shows $\mu=0$. But then there has to be another Jordan block of size $1$ with $0$ on the diagonal. This shows that for each Jordan block of size $1$ in $B$ with $\mu$ on the diagonal there has to be a different Jordan block of size $1$ with $-\mu$ on the diagonal.

Altogether we showed that for each Jordan block of size $m$ in $B$ with $0\neq \mu$ on the diagonal there is a Jordan block of size $m$ in $B$ with $-\mu$ on the diagonal and that the number of Jordan blocks of size $2m-1$ with $0$ on the diagonal is even. Thus Proposition \ref{pro:JNF} implies the statement.
\end{proof}
 
A seven-dimensional $\bF$-Lie algebra, $\bF\in \{\bR,\bC\}$, with six-dimensional Abelian ideal $\uf$ is nilpotent if and only if $\ad(e_7)|_{\uf}$ is nilpotent for $e_7\in \g\backslash \uf$ and this is the case if and only if the diagonal elements in the complex Jordan normal form are all $0$. Thus, for each partition $n_1+\ldots+n_k=6$ of $6$ with $n_1,\ldots,n_k\in \{1,\ldots,6\}$, $n_1\geq \ldots \geq n_k$ there is exactly one nilpotent Lie algebra, namely that one whose complex Jordan normal form has Jordan blocks of sizes $n_1,\ldots,n_k$, and these are all nilpotent seven-dimensional $\bF$-Lie algebras with six-dimensional Abelian ideal. Therefore, in total we have $11$ such nilpotent Lie algebras in both cases. All of them have rational structure constants so each of them admits a cocompact lattice. Hence, if $\g$ admits a cocalibrated $\G_2^{(*)}$-structure, we get a compact manifold with cocalibrated $\G_2^{(*)}$-structure. We end this article by noting what Theorem \ref{Th2real} (c), Theorem \ref{Th2complex} and Theorem \ref{Th1real} (d) imply for these nilpotent Lie algebras.
\begin{corollary}
Let $\g$ be a nilpotent $\bF$-Lie algebra of dimension seven with six-dimensional Abelian ideal $\uf$. Then:
\begin{enumerate}
\item
If $\bF=\bR$, then $\g$ admits a cocalibrated $\G_2$-structure if and only if the Jordan blocks in the complex Jordan normal form of $\ad(e_7)|_{\uf}$, $e_7\in \g\backslash \uf$, do not have the sizes $(5,1)$ or $(3,2,1)$ or $(3,1,1,1)$.
\item
If $\bF=\bR$, then $\g$ admits a cocalibrated $\G_2^*$-structure.
\item
If $\bF=\bC$, then $\g$ admits a cocalibrated $(\G_2)_{\bC}$-structure.
\end{enumerate}
\end{corollary}
\section{Acknowledgments} The author thanks the University of Hamburg for financial support and
Vicente Cort\'es and Stefan Suhr for many helpful comments on a draft version of this paper. This work
was supported by the German Research Foundation (DFG) within the
Collaborative Research Center 676 "Particles, Strings and the
Early Universe".


\begin{thebibliography}{10}
\bibitem{Be} M.\ Berger, \emph{Sur les groupes d'holonomie homog\`{e}ne des vari\'{e}t\'{e}s \`{a} connexion affine et des vari\'{e}t\'{e}s riemanniennes}, Bull. Soc. Math. France, vol. 83, (1955), 279 –- 330.
\bibitem{Br} R. Bryant, \emph{Metrics with Exceptional Holonomy}, Ann. of Math., vol. 126, no. 3, (1987), 525 -- 576.
\bibitem{BG} H.\ Busemann, D.\ E.\ Glasco II,
  \emph{Irreducible sums of simple multivectors}, Pac. J. of Math.,
  vol. 49, no. 1, (1973), 13 -- 32.
\bibitem{Ca} B. Capdevielle,
  \emph{Classification des formes trilin\'eaires altern\'ees en dimension 6},
  Enseignement Math., vol. 18, (1972), 225 –- 243.
  \bibitem{CLSS} V.\ Cort\'es, T.\ Leistner, L.\ Sch\"afer, F.\
  Schulte-Hengesbach, \emph{Half-flat structures and special
    holonomy}, Proc. London Math. Soc., vol. 102, Issue 1, (2011), 113 -- 158. 
\bibitem{GLR} I.\ Gohberg, P.\ Lancaster, L.\ Rodman, \emph{Invariant subspaces of matrices with applications}, John Wiley $\&$ Sons Inc., New York, 1986.
\bibitem{Hi} N.\ Hitchin, \emph{Stable forms and special metrics},
  Global differential geometry: the mathematical legacy of Alfred Gray
  (Bilbao, 2000), 70 -- 89.
  \bibitem{KPRS} A.\ Kasman, K.\ Pedings, A.\ Reiszl, T.\ Shiota, \emph{Universality of Rank 6 Pl\"ucker Relations and Grassmann Cone Preserving Maps}, Proc. Amer. Math. Soc., vol. 136, Issue 1, (2008), 77 -– 87.
\bibitem{MMRR} C.\ Mehl, V.\ Mehrmann, A.\ C.\ M.\ Ran, L.\ Rodman,
\emph{Eigenvalue perturbation theory of classes of structured matrices under generic structured rank one perturbations}, Linear Algebra Appl., vol. 435, Issue 3, (2011), p. 687 -- 716.
  \bibitem{MX} V.\ Mehrmann, H.\ Xu, \emph{Perturbation of purely imaginary eigenvalues of Hamiltonian matrices under structured perturbations}, Electron. J. Linear Algebra, vol. 17, (2008), 234 -- 257.
\bibitem{R} F.\ Reidegeld, \emph{Spaces admitting homogeneous $\G_2$-structures},
  Differ. Geom. Appl., vol. 28, no. 3, (2010), 301 -- 312.
  \bibitem{W1} R.\ Westwick, \emph{Irreducible lengths of trivectors of rank seven and eight},
  Pacific J. Math., vol. 80, no. 2, (1979), 575 -- 579. 
\bibitem{W2} R.\ Westwick, \emph{Real trivectors of rank seven},
  Linear Multilinear Algebra, vol. 10, Issue 3, (1981), 183 –- 204.
\end{thebibliography}
\end{document}